\documentclass[11pt]{amsart}
\usepackage{geometry}             
\usepackage{color}
\usepackage[utf8]{inputenc}
\usepackage{fontenc}
\usepackage{enumitem}
\usepackage[backref=page]{hyperref}
\usepackage[nameinlink, capitalize,noabbrev]{cleveref}
\usepackage{graphicx}
\usepackage{amssymb}
\usepackage{amsthm}
\usepackage{epstopdf}
\usepackage[all]{xy}
\usepackage[numbers]{natbib}
\definecolor{violet}{rgb}{0.0,0.2,0.7}
\definecolor{rouge2}{rgb}{0.8,0.0,0.2}
\hypersetup{
    bookmarks=true,         
    unicode=false,          
    pdftoolbar=true,        
    pdfmenubar=true,        
    pdffitwindow=false,     
    pdfstartview={FitH},    
    pdftitle={},            
    pdfauthor={},           
    colorlinks=true,        
    linkcolor=rouge2,       
    citecolor=violet,       
    filecolor=black,        
    urlcolor=cyan}          

\newtheorem{thm}{Theorem}[section]
\newtheorem{lemma}[thm]{Lemma}
\newtheorem{prop}[thm]{Proposition}
\newtheorem{cor}[thm]{Corollary}

\theoremstyle{definition}
\newtheorem{definition}[thm]{Definition}
\newtheorem{example}[thm]{Example}

\theoremstyle{remark}
\newtheorem{rem}[thm]{Remark}

\title{On the fundamental groups of effective compact Kähler orbifolds with nef anticanonical bundle}
\author{Zhining Liu}
\address{Univ Rennes, CNRS, IRMAR--UMR 6625, F-35000 Rennes, France}
\email{zhining.liu@univ-rennes1.fr}
\DeclareMathOperator{\ric}{Ric}
\DeclareMathOperator{\Ricci}{Ricci}
\DeclareMathOperator{\vol}{vol}
\DeclareMathOperator{\Aut}{Aut}
\DeclareMathOperator{\diam}{diam}
\DeclareMathOperator{\Monge}{M}

\begin{document}

\begin{abstract}
We show that the orbifold fundamental group of an effective compact Kähler orbifold with nef anticanonical bundle has polynomial growth, which generalizes M.Păun's results for manifolds \cite[Theorem~1,Theorem~2]{Paun97}
\end{abstract}
\maketitle
\tableofcontents

\section{introduction}
In \citep{Paun97}, Mihai Păun proved the following results
\begin{thm}[\protect{\citep[Theorem~1]{Paun97}}]
Let $(X,\omega)$ be a compact Kähler manifold with its anticanonical bundle $-K_X$ nef. We have the fundamental group $\pi_1(X)$ has polynomial growth.
\end{thm}

\begin{thm}[\protect{\citep[Theorem~2]{Paun97}}]
\label{opthm2}
Let $X$ be a projective manifold of dimension $n$ with $-K_X$ nef. Then the fundamental group $\pi_1(X)$ is virtually Abelian.
\end{thm}

The proof of \cref{opthm2} is by showing the Abanese morphism of $X$ is surjective then combining a result (\citep[Théorèm~2.2]{cam98}) comparing the fundamental group of $X$ and the fundamental group of $\mathrm{Alb}(X)$. It is now known that $\alpha_X:X\rightarrow \mathrm{Alb}(X)$ is surjective when $X$ is compact Kähler (\emph{cf}\citep[Theorem~1.7]{paun17}). Hence the result actually holds for compact Kähler manifolds $X$ with $-K_X$ nef.

An orbifold is a geometric object whose local model is $\mathbb{C}^n/G$ where $G$ is a finite subgroup of $\mathrm{Aut}(\mathbb{C}^n)$. As a natural generalization of manifolds, many results of manifolds whose proof only involves differential geometry are generalized to orbifold case by adapting existing proof for manifold case.

As Păun's argument used only results in complex analytic geometry and differential geometry, one may consider adapt his arguments to the orbifold case and expect to get the same results for campact Kähler orbifolds. We prove the follwoing

\begin{thm}[=\cref{Mt}, Main Theorem]
\label{mt}
 Let $(\mathcal{X},\omega)$ be an effective Kähler orbifold such that $X=|\mathcal{X}|$ is compact. If the anti-canonical bundle $K_{\mathcal{X}}^{-1}$ is nef, then $\pi_1^{\mathrm{orb}}(\mathcal{X})$ has polynomial growth.
\end{thm}

\begin{thm}[=\cref{abe}]
\label{mt2}
Let $(X,\Delta)$ be a projective orbifold pair (see \cref{pair}) with $-(K_X+\Delta)$ nef. The orbifold fundamental group $\pi_1(X,\Delta)$ is virtually Abelian.
\end{thm}


The article is organized in the following. In \cref{sec2}, we recall the preliminaries of orbifolds. In particular we recall the groupoid representation developed in \citep{Moe97} and using groupoid to construct orbi-vector bundles (\cref{vecbun}). In \cref{sec3}, we examine the metric space structure and Bishop-Gromov theorem (\cref{Gromov-Bishop})  on a Riemannian orbifold. In \cref{sec4}, we give the definition of orbifold fundamental group $\pi_1^{\mathrm{orb}}$ (\cref{pi1orb}) and use a results in \citep{Tao12} to show a version of orbifold Margulis lemma (\cref{mag}). In \cref{sec5}, we adapt Păun's argument by using the orbifold Margulis lemma to show \cref{mt}. In \cref{sec6}, we use the pair model for orbifolds to use the Albanese morphism argument as in the smooth case to obtain \cref{mt2}.

\section*{Acknowledgement}
I'm grateful for my thesis director Benoît Claudon for his constant support, for patient explaining many of my questions and for carefully reading and providing various corrections and suggestions on this article. I also thank my co-director Andreas Höring for many discussions and for insisting me typing my hand written notes to latex.

\section{preliminaries}
\label{sec2}
This section is to introduce and recall the basic notions and results in orbifolds. Through this paper, we use the classical theories of orbifolds developed mainly by Satake \cite{Sat56, Sat57}, Thursthon \cite{Thurston79} and Moerdijk and Pronk \cite{Moe97}, as they are sufficient for the main theorem. It seems more natural and trending to describe orbifolds via stacks. The reader could consult for example \cite{HM04} \cite{Pardon20vector} .

\begin{definition}
Let $X$ be a topological space, fix $n\geq 0$.
\begin{enumerate}
\item An n-dimensional real orbifold chart on X is a triple consisting of an open subset $\tilde{U}\subset \mathbb{R}^n$, a finite subgroup $G$ of $\Aut(\tilde{U})$ and a homeomorphism $\phi:\tilde{U}/G\rightarrow U$, where $U$ is an open subset of $X$. 

\item Suppose $U\subset V$ be two open subsets of $X$. A chart embedding $\lambda : (\tilde{U},G,\phi) \rightarrow (\tilde{V},H,\psi)$ is a smooth embedding $\lambda :\tilde{U}\rightarrow \tilde{V}$ such that $\psi\circ \lambda =\phi$

\item An orbifold atlas on $X$ is a family $\mathcal{U}=\{(\tilde{U},G,\phi)\}$ of orbifolds charts such that $\{U=\phi(\tilde{U})\}$ covers $X$. 
And for any $x\in X$ covered by $U$ and $V$, there exists a third orbifold chart $(\tilde{W},K,\mu)$ with $x\in W$ and two chart embedding $(\tilde{W},K,\mu)\rightarrow (\tilde{U},G,\phi)$ and $(\tilde{W},K,\mu) \rightarrow (\tilde{V},H,\psi)$.

\item An atlas $\tilde {V}$ is said to refine $\tilde {U}$ if every chart of $\tilde {V}$ embeds into some chart of $\tilde {U}$. Two atlas are said equivalent if they have a common refinement.

\item One could proceed with open subsets in $\mathbb{C}^n$ to form analytic charts and atlas
\end{enumerate}
\end{definition}

\begin{definition}
\label{o1}
A real (\emph{resp.} complex) effective orbifold $\mathcal{X}$ of dimension $n$ is a collection of the following data:
\begin{enumerate}[label=(\roman*)]
\item A topological space $X$ which is Hausdorff and second countable;
\item An equivalence class $[\mathcal{U}]$ of real (\emph{resp.} complex) $n$-dimensional orbifold atlas.
\end{enumerate}
\end{definition} 

In the following, we often denote without tilde an open subset of the underline space of an orbifold and with tilde for the corresponding open subset of its orbifold charts.

\begin{example}
\leavevmode
\begin{enumerate}[label=\arabic*), start=0]
\item Any manifold is an orbifold with its manifold charts together with trivial group serving as orbifold charts;

\item $H:=\{(x,y)\in \mathbb{R}^2: x\geq 0\}$ can be given an orbifold structure by consider $\phi:\mathbb{R}^2
\rightarrow H$, $(x,y) \mapsto (|x|,y)$. The group action $\mathbb{Z}/2\mathbb {Z}$ on $\mathbb{R}^2$ is given by $(x,y)\mapsto (-x,y)$;

\end{enumerate}
\end{example}

\begin{definition}
A smooth (resp. holomorphic) map $f$ between to orbifolds $\mathcal{X}=(X,\mathcal{U})$ and $\mathcal{Y}=(Y,\mathcal{V})$ is a continuous map on the underline spaces, i.e. $f:X \rightarrow Y$, such that for any $x\in X$ if we denote $y=f(x)$, there exists a chart $(\tilde{U},G,\phi)$ for $x$ and a chart $(\tilde{V},H,\psi)$ for $y$ and a smooth(resp. holomorphic) map $\tilde{f}:\tilde{U}\rightarrow \tilde{V}$ such that the following diagram commutes 
\begin{center}
$\xymatrix{\tilde{U}\ar[r]^{\tilde{f}}\ar[d]&\tilde{V}\ar[d]\\ U\ar[r]^f&V}$.
\end{center}
\end{definition}

Though with the name 'smooth', we note that smooth map does not behaves well. For example, we don't know if $f:\mathcal{X}\rightarrow\mathcal{Y}$ is smooth, then $f$ induces pullback morphism of differential forms. And the author does not know if for $x\in |X|$, $f$ induces a morphism of local groups $G_x\rightarrow G_{f(x)}$. To overcome this problem, we use the notion of "strong morphism" introduced by Moerdijk and Pronk \cite{Moe97}.

\citep{Moe97} deals with orbifolds by identifying them with certain groupoids and define the maps between orbifolds to be the ones induced by morphisms between groupoids.

\begin{definition}
A topological groupoid $\mathcal{G}$ consists of a topological space $G_0$ of objects and a topological space $G_1$ of arrows, together with five continuous structure maps listed below:
\begin{enumerate}[label=\arabic*.]
\item The source map $s:G_1\rightarrow G_0$, which assigns to each arrow $g\in G_1$ its source $s(g)$.
\item The target map $t:G_1\rightarrow G_0$, which assigns to each arrow $g\in G_1$ its target $t(g)$. For any two objects $x,y\in G_0$, one writes $g:x\rightarrow y$ to indicate that $g\in G_1$ is an arrow with $s(g)=x$ and $t(g)=y$.
\item The composition map $m:G_1$ $_s\times_t$ $G_1\rightarrow G_1$. If $h:y\rightarrow z$, $g:x \rightarrow y$, then $hg=m(h,g):x\rightarrow z$. $m$ is required to be associative.
\item The identity map $u:G_0\rightarrow G_1$ which is a two-sided unit for the composition.
\item The inverse map $i:G_1\rightarrow G_1$. If $g:x\rightarrow y\in G_1$, then $g^{-1}=i(g):y\rightarrow x$ is the two sided inverse to $g$, i.e. ,$g\circ i(g)=u(y)$ and $i(g)\circ g=u(x)$.
\end{enumerate}
 \end{definition}
 
 One can also consider the topological groupoid as a groupoid (i.e. a category whose morphisms are all isomorphisms) equipped with two topological structure on the sets of objects and morphisms such that the structural maps are continuous. We then define the
 \begin{definition}
 A Lie groupoid is a topological groupoid $\mathcal{G}$ where $G_0$ and $G_1$ are smooth manifolds and all the structure maps $s,t,m,u$ and $i$ are smooth. Furthermore, $s$ and $t$ are required to be submersions (hence $G_1$ $_s\times_t$ $G_1$ is a manifold).
  \end{definition}
 With the topological/smooth structures on $G_0$ and $G_1$, one can define the morphisms between groupoids to be functors with required continuity/smoothness on the sets of obejects and morphisms. Similarly, a groupoid natural transformation will require the continuity on the functions assign each objects in source groupoid to the morphisms in the target groupoid.

 \begin{example}
 Let $M$ be a topological space $K$ a topological group acting on $M$. One defines a topological groupoid $K\ltimes M$, by setting $(K\ltimes M)_0=M$ and $(K\ltimes M)_1=K\times M$, with $(g,x):x\rightarrow gx$. This groupoid is called the action groupoid or translation groupoid associated to the group action. Note if $M$ is a manifold and $K$ a Lie group, then $K\ltimes M$ becomes a Lie groupoid. If moreover, $K=\{e\}$, then the associated groupoid (called the unit groupoid) is nothing but a manifold.
 \end{example}
 
 \begin{definition}
Let $\mathcal{G}$ be a Lie groupoid. For a point $x\in G_0$, we define the isotropy group $G_x$ of $\mathcal{G}$ to be $s^{-1}(x)\cap t^{-1}(x)$. And we define the orbit space $|\mathcal{G}|$ of $\mathcal{G}$ to be the quotient of $G_0$ by the equivalence relation $x\sim y$ iff $\exists\, g:x\rightarrow y$.
\end{definition}

\begin{lemma}[\emph{cf.} \protect{\cite[Corollary~1.4.11]{Mac05}}]
Let $\mathcal{G}$ be a Lie groupoid. Set $G_x^y$ being $s^{-1}(x)\cap t^{-1}(y)$. Then $G_x^y$ is a smooth manifold, and the morphism $m:G_y^z\times G_x^y\rightarrow G_x^z$ is smooth. In particular, $G_x$ is a Lie group.
\end{lemma}

Now we define types of groupoids.
\begin{definition}
Let $\mathcal{G}$ be a Lie groupoid.
\begin{enumerate}[label=(\alph*)]
\item$\mathcal{G}$ is proper, if $(s,t):G_1\rightarrow G_0\times G_0$ is proper.
\item $\mathcal{G}$ is called a foliation groupoid if each isotropy group $G_x$ is discrete.
\item $\mathcal{G}$ is étale, if $s$ and $t$ are local diffeomorphisms. In this case, one defines the dimension of $\mathcal{G}$ to be $\dim(\mathcal{G}):=\dim(G_0)=\dim(G_1)$.
\end{enumerate}
\end{definition}

A direct observation is that when $\mathcal{G}$ is proper étale, $G_x^y$ is finite. Now if $g:x\rightarrow y$ then as $s$ and $t$ are diffeomorphisms around $g,\ x$ and $y$ we get (\emph{via} $t\circ s^{-1}$) a local diffeomorphism $\phi_g: U_x \rightarrow U_y$. After shrinking $U_x$ and $U_y$, we get a morphism $\phi: G_x^y \rightarrow \mathrm{Diff}(U_x,U_y)$. One can prove that $\phi(hg)=\phi(h)\circ\phi(g)$. In particular, $\phi:G_x\rightarrow$Diff$(U_x)$ is a group morphism.

\begin{definition}
We define an orbifold groupoid to be a proper étale Lie groupoid. An orbifold groupoid $\mathcal{G}$ is effective if $\forall\, x\in G_0$, $\phi:G_x\rightarrow\mathrm{Diff}(U_x)$ is injective.
\end{definition}

\begin{definition}
A morphism $\phi:\mathcal{H}\rightarrow\mathcal{G}$ between Lie groupoids is called an equivalence if both conditions below are satisfied:
\begin{enumerate}[label=(\roman*)]
\item the map
\[t\pi _1:G_1\,{}_s\times_{\phi}\,H_0\rightarrow G_0\]
defined on the fibered product of manifolds is a surjective submersion;
\item $(H_1,\phi_1,(s,t))$ is a fibered product of $\phi \times \phi:H_0\times H_0 \rightarrow G_0\times G_0$ and $(s,t):G_1\rightarrow G_0\times G_0$.
\end{enumerate}
\end{definition}
Note that a homomorphism $\phi:\mathcal{H}\rightarrow\mathcal{G}$ induces a continuous map $|\mathcal{H}|\rightarrow|\mathcal{G}|$. When $\phi$ is an equivalence, the induced map on orbit spaces is an homoemorphism.

Here is Moerdijk and Pronk's definition (\emph{cf.} \cite[Theorem~4.1.]{Moe97} \cite[Definition~1.48.]{Ruan07}) of  orbifolds.
\begin{definition}\label{o2}
\leavevmode
\begin{enumerate}
\item Two Lie groupoids $\mathcal{G}$ and $\mathcal{G}'$ are said to be Morita equivalent, if there exists a third groupoid $\mathcal{H}$ and two equivalences
\[\mathcal{G} \leftarrow \mathcal{H}\rightarrow \mathcal{G}'.\]
\item An orbifold structure on a paracompact Hausdorff space $X$ consists of an orbifold groupoid $\mathcal{G}$ and a homeomorphism $f:|\mathcal{G}|\rightarrow X$. We say two orbifold structures $f:|\mathcal{G}| \rightarrow X$ and $g:|\mathcal{H}|\rightarrow X$ are equivalent, if there exists an equivalence of orbifolds $\phi:\mathcal{G}\rightarrow \mathcal{H}$ such that $f=g\circ|\phi|$.
\item An orbifold $\mathcal{X}$ is a space $X$ together with a class of equivalent orbifold structure. An element $f:|\mathcal{G}|\rightarrow X$ is called a presentation of the orbifold $\mathcal{X}$.
\end{enumerate}
\end{definition}

It turns out that the local structure of an orbifold groupoid $\mathcal{G}$ around $x\in G_0$ is completely determined by the local group $G_x$. More precisely, we have the following
\begin{prop}
\label{local}
Let $\mathcal{G}$ be an orbifold groupoid. $\forall x\in G_0$, and for any neighbourhood $U$ of $x$, there exists an open neighbourhood $N_x\subset U$, such that the restriction of $\mathcal{G}$ over $N_x$ is isomorphic, as Lie groupoid, to the translation groupoid $G_x\ltimes N_x$ and the quotient space $N_x/G_x$ is an embedded open subsets of $|\mathcal{G}|$ via the natural morphism $\mathcal{G}\vert_{N_x}\hookrightarrow \mathcal{G}$.
\end{prop}
\begin{proof}
For the proof, see \cite[Proposition~1.44]{Ruan07}.
\end{proof}

Let $(X,\mathcal{G},f:|\mathcal{G}|\rightarrow X)$ be an orbifold in the sense of Definition~1.10. Take $x \in X$ and $\tilde{x}\in G_0$ one of its pre-image. By Proposition~1.1., we give an orbifold chart $(U_{\tilde{x}},G_{\tilde{x}})\rightarrow f(U_{\tilde{x}}/G_{\tilde{x}})$ around $x$. It's easy to see that we could get an atlas consisting of all these charts. Hence we get an orbifold $(X,\tilde{U})$ in the sense of \cref{o1}. We call $\mathcal{U}$ the orbifold atlas associated to $\mathcal{G}$. If $\mathcal{G}$ is Morita equivalent to $\mathcal{H}$, then their associated atlases are equivalent.
 In \cite{Moe97}, Moerdijk and Pronk prove that for any orbifold $(X,\tilde{U})$ in the sense of \cref{o1}, we could get a unique (up to equivalence) groupoid representation of $X$.

\begin{thm}
Let $\mathcal{X}=(X,\mathcal{U})$ be an orbifolds in the sense of \cref{o1}. There exists up to a Morita equivalence, a unique effective orbifold groupoid $\mathcal{G}$ and a homeomorphism $|\mathcal{G}|\rightarrow X$, such that the associated atlas $\mathcal{V}$ of $\mathcal{G}$ is equivalent to $\mathcal{U}$.
\end{thm}

Thus we may interchange freely both definitions of orbifolds: in terms of atlas or in terms of groupoid. We now give the defintion of strong maps.

\begin{definition}
Let $\mathcal{X},\mathcal{Y}$ be two orbifolds. A strong map $f$ from $\mathcal{X}$ to $\mathcal{Y}$ is a continuous map $f:|\mathcal{X}|\rightarrow|\mathcal{Y}|$ such that there are representations $\mathcal{G},\mathcal{H}$ of $\mathcal{X}$ and $\mathcal{Y}$ respectively and a groupoid morphism $F:\mathcal{X}\rightarrow \mathcal{G}$ such that the following diagram commutes
\begin{center}
$\xymatrix{|\mathcal{G}| \ar[d]\ar[r]^{|F|}& |\mathcal{H}|\ar[d]\\ |\mathcal{X}|\ar[r]^f &|\mathcal{Y}|}$
.
\end{center}
\end{definition}

A strong map is clearly smooth. From the above definition, we see that $f$ induces morphisms between local groups.

One of the advantage of \cref{o2} is to give a simple way to define fiber bundles and vector bundles. For details, see \cite[Chapter~2]{Ruan07}. We only define vector bundles.
\begin{definition}
Let $\mathcal{G}$ be a topological groupoid, $G_1$ its arrows and $G_0$ its objects. A left-$\mathcal{G}$-vector bundle is a triple $(E,\pi,\mu)$, where $\pi:E\rightarrow G_0 $ is a vector bundle over $G_0$ and $\mu:G_1$ $_s\times_{\pi}E\rightarrow E$ is a continuous map, satisfying the following
\begin{enumerate}
\item $\mu(g'g,e)=\mu(g',\mu(g,e))$;
\item $\mu(1,e)=e$;
\item $\mu(g,-):E_{s(g)}\rightarrow E_{t(g)}$ is a linear isomorphism.
\end{enumerate}
\end{definition}

Let $E$ be a left-$\mathcal{G}$-vector bundle over $\mathcal{G}$. We may associate with $E$ a groupoid $\mathcal{G}\ltimes E$ in the following way: take $E$ to be its objects and $G_1$ $_s\times_{\pi}E$ be its arrows. The source map is $(g,e)\mapsto e$ and target map is $(g,e)\mapsto \mu(g,e)$. If $e\in E_x$, the identity arrow of $e$ is $(1_x,e)$. The inverse arrow of $(g,x)$ is $(g^{-1},gx)$. $\pi:E\rightarrow G_0$ extends to a morphism $\mathcal{G}\ltimes E\rightarrow \mathcal{G}$ by taking the objects map $\pi$ and the arrows map $(g,x)\mapsto g$. We may also note $\mathcal{G}\ltimes E$ by $\mathcal{E}$

Let $\mathcal{H}$ be another groupoid and $\phi:\mathcal{H}\rightarrow \mathcal{G}$ a morphism between groupoid. Consider the pullback vector bundle $\phi_0^{\ast}(E)\rightarrow H_0$. It has a natural left-$\mathcal{H}$-vector bundle structure $\nu:H_1$ $_s\times \phi_0^{\ast}(E)\rightarrow \phi_0^{\ast}(E)$, by defining $\nu(h,(x,e))=\mu(\phi_1(h),e)$. Its associated groupoid $\mathcal{H}\ltimes \phi_0^{\ast}(E)$ fits in the following commutative diagram 
\begin{center}
$\xymatrix{\mathcal{H}\ltimes \phi_0^{\ast}(E) \ar[r]\ar[d]&\mathcal{E}\ar[d]\\\mathcal{H}\ar[r]&\mathcal{G}}$.
\end{center}
Thus it makes sense to call $\mathcal{H}\ltimes \phi_0^{\ast}(E)$ the pullback of $\mathcal{E}$ and note it also by $\phi^{\ast}(\mathcal{E})$.

Suppose now we have $\phi:\mathcal{H}\rightarrow \mathcal{G}$ be an equivalence between two orbifold groupoids. Then $\phi$ as a functor is an equivalence. Applying \cref{local} to both $\mathcal{H}$ and $\mathcal{G}$, one sees this indicated that $\phi_0$ is a local diffeomorphism.  Let $E$ be a left-$\mathcal{H}$-vector bundle, $x\in G_0$, $G_1\ni g_1:x\rightarrow x_1$, and $G_1\ni g_2: x\rightarrow x_2$ such that $x_1=\phi_0(y_1)$, $x_2=\phi_0(y_2)$. We know that there exists a unique $h\in H_1$ such that $\phi_1(h)=g_1g_2^{-1}$. As $\phi_0$ is a local diffeomorphism, we could, via pullback, get $a$ $priori$ two vector bundles around $x$. Denote $\psi_i$, $i=1,2$ the local inverse of $\phi_0$ at $y_i$. We have vector bundles $g_1^{\ast}\psi_1^{\ast}(E)$ and $g_2^{\ast}\psi_2^{\ast}(E)$. They are indeed canonically  isomorphic. Since $E$ is a $\mathcal{H}$ vector bundle, we have $h^{\ast}E=E$ and
\begin{equation*}
g_2^{\ast}\psi_2^{\ast}(E)=g_2^{\ast}\psi_2^{\ast}h^{\ast}(E)=g_2^{\ast}(g_1 g_2^{-1})^{\ast}\psi_2^{\ast}(E)=g_1^{\ast}\psi_1^{\ast}(E).
\end{equation*} 
Thus for any point $x\in G_0$ there is a well-defined vector bundle $(\phi_0^{-1})^{\ast}(E)$. It's easy to see they glue up to a vector bundle on $G_0$, and from the construction, this vector bundle is actually a left-$\mathcal{G}$-vector bundle. We call it the pushfoward $\phi_{\ast}(\mathcal{E})$ of $\mathcal{E}$. We have $\phi^{\ast}\phi_{\ast}(\mathcal{E})=\mathcal{E}$. If either $\mathcal{H}$ or $\mathcal{G}$ fails to be orbifold groupoid, the $\phi_0$ fails in general to be local diffeomorphism. This happens, for example when we consider the Morita equivalence $e:\mathcal{G}\leftarrow \mathcal{K}\rightarrow \mathcal{H}$. Even $\mathcal{G}$ and $\mathcal{H}$ are orbifold groupoids, $\mathcal{K}$ needs not to be an orbifold groupoid. Hence the above explicit construction can not be used on Morita equivalence. However, we have 

\begin{prop}[\emph{cf.} \protect{\cite[page~11 Remark ~(4)]{Moe97}}]
Let $e:\mathcal{G}\leftarrow \mathcal{K}\rightarrow \mathcal{H}$ be a Morita equivalence between two topological groupoids. Then $e$ induces an equivalence $e^{\ast}:Sh(\mathcal{H})\rightarrow Sh(\mathcal{G})$ between two toposes.
\end{prop}

We may now form our definition of orbi-vector bundles on orbifolds.
\begin{definition}
\label{vecbun}
Let $\mathcal{X}$ be an orbifold. A real (resp. complex) vector bundle of rank $r$ is a strong map $p:\mathcal{V}\rightarrow \mathcal{X}$ together with the following
\begin{enumerate}
\item A representation $f:\mathcal{G}\rightarrow |\mathcal{X}|$ of $\mathcal{X}$;
\item A real (resp, complex) left-$\mathcal{G}$-vector bundle $E$ on $G_0$;
\item A homeomorphism $g:\mathcal{G}\ltimes E\rightarrow |\mathcal{V}|$, such that $g$ gives a representation of the orbifold $\mathcal{V}$ and $\mathcal{G}\ltimes E\rightarrow \mathcal{G}$ represents $p$
\end{enumerate} 

\end{definition}

As in the manifold case, one need another definition for holomorphic orbi-vector bundles. Our holomorphic bundle is defined in \cref{hol}. There is also other issues on how to define sections, see \cref{rhol}. For a better definition, one could see the vector bundles defined via representable $2$-functors as in \cite{Pardon20vector}.

Let $\mathcal{V}\rightarrow \mathcal{X}$ be an orbi-vector bundle represented by $\mathcal{E}=\mathcal{G}\ltimes E\rightarrow \mathcal{G}$. Let $G_x\ltimes U_x\cong \mathcal{G}|_{U_x}$ be as in \cref{local} for a $x\in G_0$. We may take $U_x$ sufficiently small such that there is a trivialization $E|_{U_x}\cong U_x\times \mathbb{F}^r$. Then there is an isomorphism $\mathcal{E}|_{U_x}\cong G_x\ltimes (U_x\times \mathbb{F}^r)$. The actions of $G_x$ fits into a commutative diagram
\begin{center}
$\xymatrix{G_x\times (U_x\times \mathbb{F}^r)\ar[d]\ar[r]&U_x\times \mathbb{F}^r \ar[d] \\G_x\times U_x\ar[r] &U_x}$.
\end{center}
Hence we have $(U_x\times \mathbb{F}^r)/G_x$ and $U_x/G_x$ be the orbifold charts of $\mathcal{V}$ and $\mathcal{X}$ respectively. Note for any $y\in U_x/G_x$, its fiber $|\mathcal{V}|_y$ is isomorphic to $\mathbb{F}^r/G'$, where $G'$ is a subgroup of $G_x$.

\begin{definition}
Let $\pi:E\rightarrow G_0$ be a left-$\mathcal{G}$-vector bundle. A $\mathcal{G}$-section of $E$ over $U\subset G_0$ is a section $s:U\rightarrow E$ of $\pi$ such that for any $g\in G_1$, we have $g\cdot s(x)=s(g\cdot x)$. Note $s$ induces a morphism $\mathcal{G}|_U \rightarrow (\mathcal{G}\ltimes E)|_{U}$.

If the orbi-vector bundle $p:\mathcal{V}\rightarrow \mathcal{X}$ is represented by $\mathcal{G}\ltimes E\rightarrow \mathcal{G}$, we say $\sigma:U/\mathcal{G}\rightarrow \mathcal{V}$ a section of $p$ if $\sigma$ corresponds to a $\mathcal{G}$-section $s$ of $E$.
\end{definition}

\begin{example}
If $E$ and $F$ are left-$\mathcal{G}$-bundles for the topological groupoid $\mathcal{G}$, $E\otimes F$, $E\wedge F$, $\mathrm{Sym}^n(E)$, $\mathrm{Hom}(E,F)$ and $E^{\vee}$ have natural left-$\mathcal{G}$-bundle structures, and we denote their associated groupoids by $\mathcal{E}\otimes\mathcal{F}$, $\mathcal{E}\wedge \mathcal{F}$, $\mathrm{Sym}^n(\mathcal{E})$, $\mathrm{Hom}(\mathcal{E},\mathcal{F})$ and $\mathcal{E}^{\vee}$ respectively.
\end{example}

\begin{example}
\label{tensor-bundle}
Let $\mathcal{X}$ be a real orbifold of dimension $n$. Suppose $\mathcal{U}$ be its atlas. For each chart $(\tilde{U_i},G_i,\phi_i)$, we associate a $G_i$-space $(T\tilde{U_i},G_i)$, where $g_i$ acts on $T\tilde{U_i}$ by its tangent associated to its action on $\tilde{U_i}$. We identify $[v_i]\in T\tilde{U_i}/G_i$ and $[v_j] \in T\tilde{U_j}/G_j$ if there are orbifold chart embeddings $\rho_i: \tilde{V}\rightarrow \tilde{U_i}$ and $\rho_j: \tilde{V}\rightarrow \tilde{U_j}$ such that $T\rho_i(v_i)=T\rho_j(v_j)$. The topological space $\{\bigsqcup_i \tilde{U_i}/G_i\}/([v_i]\sim [v_j])$ has an orbifold atlas, whose charts are $(T\tilde{U_i},G_i)$. We denote this orbifold $T\mathcal{X}$. Note by construction we have a natural projection $T\mathcal{X}\rightarrow \mathcal{X}$. If $f:|\mathcal{G}|\rightarrow |\mathcal{X}|$ is a representation of $\mathcal{X}$, then $T\mathcal{G}:=\mathcal{G}\ltimes TG_0$ is a representation of $T\mathcal{X}$ and the projection $T\mathcal{X}\rightarrow \mathcal{X}$ is represented by $T\mathcal{G}\rightarrow \mathcal{G}$.
\end{example}

Hence when considering tangent bundles and cotangent bundles of $\mathcal{X}$, we don't distinguish the representation groupoids $\mathcal{G}$ used to construct the representation $T\mathcal{G}$ of $T\mathcal{X}$.

For the construction of tangent bundle $T\mathcal{X}$ of $\mathcal{X}$, we see that giving a $p$-form $\omega$ over $O\subset X$ an open subset of $X=|\mathcal{X}|$ is equivalent to find a cover of $O$ by orbifold charts $(U_i,G_i)$ and $G_i$-invariant $p$-forms $\omega_i$ on $U_i$ such that for any chart embedding $\lambda:U_i\rightarrow U_j$, one has $\lambda^{\ast}(\omega_j)=\omega_i$. One can say the same for $p$-tensors (but not for currents).

\begin{definition}
Let $\mathcal{X}$ be a real orbifold of dimension n. we say that $\mathcal{X}$ is orientable if there exists a non-vanishing n-form $\alpha$ on $\mathcal{X}$. We say a chart $(\tilde{U},G,\phi)$ is compatible with this orientation if $\phi^{\ast}(\alpha)=\lambda\cdot vol_{\mathbb{R}^n}$, $\lambda$ is a positive function. For $\omega$ an n-form supported in $U$, 
we define its integration $\int_{\mathcal{X}} \omega :=\frac{1}{|G|}\int _{\tilde{U}} \phi^{\ast}\omega$. For general case, we cover $\mathcal{X}$ by charts $(\tilde{U_i},G_i,\phi_i)$, take a partition of unity $\rho_i$ with respects to $\{U_i\}$, and define
\begin{equation*}
\int_{\mathcal{X}} \omega:=\sum_i \int _{\mathcal{X}} \rho_i \omega .
\end{equation*}
\end{definition}

Suppose $\omega$ is supported in $U$, and $(\tilde{U}_i,G_i,\phi_i)$ be charts that embedded via $\lambda_i$ to the chart $(\tilde{U},G,\phi)$. (\emph{cf.} \cite[page~5 Remark~(6)]{Moe97}) we have $G_i$ is a subgroup of $G$ and  all the distinct embedding of $(\tilde{U}_i,G_i,\phi_i)$ into $(\tilde{U},G,\phi)$ will be $g\cdot 
\lambda_i$, where $g\cdot G_i$ forms the cosets. Thus 

\begin{equation*}
\begin{split}
\int_\mathcal{X} \omega &=\frac{1}{|G|}\int_{\tilde{U}} \phi^{\ast}\omega\\
&=\frac{1}{|G|}\sum_i \int_{\tilde{U}}(\rho_i \circ \phi )\phi^{\ast}\omega\\
&=\sum_i\sum_{g}\frac{1}{G}\int _{g\lambda(\tilde{U}_i)}(\rho_i\circ \phi )\phi^{\ast}\omega\\
&=\sum_i\frac{1}{G_i} \int_{\tilde{U}_i}(\rho_i\circ\phi_i)\phi_i^{\ast}\omega
\end{split}
\end{equation*}
where in the forth equation $g$ runs through the representatives of cosets $G/G_i$. For the general case, if there are two covering $\tilde{U}_i$, $\tilde{V}_j$ we take a third covering $\tilde{W}_k$ such that each $\tilde{W}_k$ embeds into some $\tilde{U}_i$ and $\tilde{V}_j$. Then it reduces the argument to a single chart and its refinement. Integration is thus well defined.

Let $\mathcal{A}^p(\mathcal{X})$ be the global differential $p$-forms on $\mathcal{X}$. We see that the exterior differential $d$ maps $\mathcal{A}^p$ to $\mathcal{A}^{p-1}$. Hence, it makes sense to consider the de Rham cohomolology of $\mathcal{X}$. We recall some basic results.

\begin{prop}
Let $\mathcal{X}$ be a real $n$-orbifold, $X=|\mathcal{X}|$. We have a canonical isomorphism $H^p(X,\mathbb{R})\cong H^p_{dR}(\mathcal{X})$.
\end{prop}
\begin{proof}
We give a direct proof. Consider the sheaf $C^{\infty}_{\mathcal{X}}$ on $X$, given by $V\mapsto \mathrm{Mor}(V,\mathbb{R})$. Note $\mathcal{X}$ has partition of unity by smooth function with respect to any open cover. We have  $\mathcal{A}^p$ is fine and acyclic (\emph{cf.} \cite[Definition~4.35 and Proposition~4.36]{vo02}). On the other hand, $\mathcal{A}^{\bullet}$ is a resolution of $\mathbb{R}_X$. Hence we have the canonical isomorphism $H^p(X,\mathbb{R})\cong H^p_{dR}(\mathcal{X})$.
\end{proof}

In \cite[section~7]{Sat56}, Satake shows there is a canonical morphism 
\begin{center}
$H^{sing}_p(\mathcal{X},\mathbb{R})\rightarrow \check{H}_p(\mathcal{U},\mathbb{R})$,
\end{center}
where the latter is the Čech homology group. If we define $\check{H}_p(X,\mathbb{R})=\varprojlim\check{H}_p(\mathcal{U},\mathbb{R})$. Then we have
\begin{prop}[\emph{cf.} \protect{\cite[Theorem~2]{Sat56}}]
Let $\mathcal{X}$ be a real $n$-orbifold, $X=|\mathcal{X}|$. We have a canonical isomorphism  $H^{sing}_p(\mathcal{X},\mathbb{R})\rightarrow \check{H}_p(X,\mathbb{R})$
\end{prop}

As $\check{H}^p(\mathcal{U},\mathbb{R})$ is dual to $\check{H}_p(\mathcal{U},\mathbb{R})$, we see the isomorphism $H^{n-p}_{dR}(\mathcal{X})\cong H_p^{sing}(X,\mathbb{R})$. Moreover, we have the Poincaré duality for orbifolds:
\begin{prop}
Let $\mathcal{X}$ be a compact real $n$-orbifold, the natural map 
\[
\left\{\begin{array}{ccl}
H^p_{dR}(\mathcal{X})\times H^{n-p}_{dR}(\mathcal{X}) &\longrightarrow &\mathbb{R}\\
(\omega,\theta)&\mapsto &\int_\mathcal{X}\omega \wedge \theta\end{array}\right.\]
is a perfect paring.
\end{prop}


\section{riemannian orbifolds}
\label{sec3}
From this section onward, all orbifolds are understood to be effective.

\subsection{Differential calculus on orbifold}

\begin{definition}
A Riemmanian orbifold is a pair $(\mathcal{X},g)$ where $\mathcal{X}$ is an orbifold and $g$ is an (equivariant) section of $(T^2\mathcal{X})^{\vee}$ such that the following equivalent condition is satisfied:
\begin{enumerate}[label=(\roman*)]
\item If $\mathcal{X}$ is represented by $\mathcal{G}$, and $g$ corresponds 
to $\sigma:G_0\rightarrow (T^2{G_0})^{\vee}$ then $\sigma$ is a Riemannian metric on $G_0$;
\item There exists a family of charts $(\tilde{U_i},G_i)$ with $G_i$-invariant metrics $\tilde{g}_i$ represent $g$.
\end{enumerate}
\end{definition}

Most operators on Riemannian manifolds can be generalized to Riemannian 
orbifolds. We begin to treat some basic results on covariant derivatives on orbifolds.

\medskip
Let $\mathcal{X}$ be an orbifold, $\{\mathcal{G},f:|\mathcal{G}|\rightarrow |\mathcal{X}|\}$ being a groupoid representation of $\mathcal{X}$. We know from \cref{tensor-bundle}. that $T^p\mathcal{X}\otimes T^q(T\mathcal{X}^{\vee})$ is represented by $T^p\mathcal{G}\otimes T^q(T\mathcal{G}^{\vee})$. A 
$(p,q)$-tensor over an open subset $U$ of $|\mathcal{X}|$ is thus a collection of $G_i$-invariant $(p,q)$-tensor over $\tilde{U}_i$ such that $U_i=\tilde{U}_i/G_i$ cover $U$. For $T=X_1\otimes X_2 \otimes \cdots\otimes X_p\otimes S$, we have
\begin{equation*}
g\cdot T=g_{\ast}(X_1)\otimes g_{\ast}(X_2)\otimes\cdots\otimes g_{\ast}(X_p)\otimes (g^{-1})^{\ast}(S).
\end{equation*}

As all the calculation can be performed locally, in the following we consider 
a local model where $U\subset \mathbb{R}^n$ is an open subset, $H$ a finite subgroup of $\mathrm{Aut}(U)$, and $g$ a Riemannian metric on $U$ which is $H$-invariant. 
For any $h\in H$, we note the action of $h$ on $x\in U$ by $L_h(x)$ or $h\cdot x$. Also for any smooth function $f$, we define the $H$-action on $f$ by $h\cdot f=f\circ L_{h^{-1}}$. A easy consequence for this adaption is that for any $(p,q)$-tensor $T$, we have $h\cdot (fT)=(h\cdot f)(h\cdot T)$.

Now consider $T$ be a $(0,p)$-tensor over $U$, $X_i, 1\leq i\leq p$ p vector fields over $U$. We have 
\begin{equation*}
\begin{split}
h^{\ast}(T)(X_1,X_2,\ldots,X_p)(x)
&=T(hx)(T_xL_h X_1(x),\ldots, T_xL_hX_p(x))\\
&=T(hx)(h_{\ast}X(xh),\ldots, h_{\ast}X(xh))\\
&=T(h_{\ast}X_1,\ldots,h_{\ast}X_p)(hx)\\
\end{split}
\end{equation*}

Thus we can characterize $T$ being $H$-invariant by 
\begin{equation*}
T(X_1,X_2,\ldots,X_p)=T(h_{\ast}(X_1),\ldots,h_{\ast}(X_p))\circ L_h
\end{equation*}
for any $h\in H$ and any vector fields $X_1,...X_p$. 

The metric $g$ being $H$-invariant, we infer that $L_h$ is an isometry for any $h\in H$ and we have $h_{\ast}(\nabla_X Y)=\nabla_{(h_{\ast}X)}(h_{\ast}Y)$ where $\nabla$ is the Levi-Civita connection of the metric $g$.

To convince the reader that most conceptions in Riemannian manifolds generalize to Riemannian orbifolds. We show the the following
\begin{lemma}
\label{invariance}
If $T$ is a $(0,p)$-tensor which is $H$-invariant, we have the $(0,p+1)$-tensor $\nabla T$ is $H$-invariant.
\end{lemma}
\begin{proof}
 Let $X_0,\ldots,X_p$ be $p+1$ vector fields. We compute 
\begin{align*}
((\nabla T)&(h_{\ast}(X_0),h_{\ast}(X_1),\ldots,h_{\ast}(X_p)))\circ L_h\\
&=((\nabla_{h_{\ast}(X_0)}T)(h_{\ast}(X_1),\ldots,h_{\ast}(X_p)))\circ L_h\\
&= \Big((h_{\ast}(X_0)\cdot T(h_{\ast}(X_1),\ldots,h_{\ast}(X_p)\\
&\hspace*{1cm}-\sum_i T(h_{\ast}(X_1),\ldots,\nabla_{h_{\ast}(X_0)}h_{\ast}(X_i),\ldots,h_{\ast}(X_p))\Big)\circ L_h\\
&=\Big((X_0\cdot (T(h_{\ast}(X_1),\ldots,h_{\ast}(X_p))\circ L_h))\circ L_{h^{-1}}\\
&\hspace*{1cm}-\sum_i T(h_{\ast}(X_1),\ldots,h_{\ast}(\nabla_{X_0}X_i),\ldots,h_{\ast}(X_p))\Big)\circ L_h\\
&=X_0\cdot T(X_1,\ldots,X_p)-\sum_i T(X_1,\ldots,\nabla_{X_0} X_i,\ldots, X_p)\\
&=(\nabla T)(X_0,X_1,\ldots,X_p).
\end{align*}
\end{proof}

As differential commutes with pull-back, we see that if $\omega$ is an invariant $p$-form, so is $d\omega$. In particular, if $f$ is $H$-invariant smooth function, the $df$ is an invariant $1$-form and $\nabla f=(df)^{\sharp}$ is an invariant vector field.

Suppose now we have two orbi-vector filed $W,V$ over $\mathcal{X}$. Take $(\tilde{U},G)$ a chart for $\mathcal{X}$ such that $W,V$ are represented 
by $G$-invariant field $\tilde{W},\tilde{V}$ respectively. Then we have $h_\ast(\nabla_{\tilde{W}}\tilde{V})=\nabla_{\tilde{W}}\tilde{V}$ for any $h\in G$. If $\lambda:(\tilde{U}',G') \rightarrow (\tilde{U},G)$ is an chart embedding, and $\tilde{W}',\tilde{V}'$ are the representations of $W,V$ on $\tilde{U}'$, then $\nabla_{\tilde{W}'}\tilde{V}'=\lambda^{\ast}(\nabla_{\tilde{W}}\tilde{V})$. Thus all the local representations glue back to an orbi-vector field. We may thus define:
\begin{definition}
Let $W,V$ be two orbi-vector fields over $(\mathcal{X},g)$, represented by $\tilde{W}_i,\tilde{V}_i$ on a covering $(\tilde{U}_i,G_i)$ respectively. Let $\nabla_i$ be the Levi-Civita connection on $\tilde{U}_i$, then there is a unique vector filed $\nabla_W V$ on $\mathcal{X}$ corresponding to the family $\nabla_i {}_{\tilde{W}_i}{\tilde{V}_i}$. We define the association $\nabla:W,V\mapsto \nabla_W V$ as the Levi-Civita connection on $(\mathcal{X},g)$.
\end{definition}

If $\tilde{R}_i$ is the curvature tensor of $(\tilde{U}_i,\tilde{g}_i)$, we may glue them to an orbi-tensor $R$. 
We call this tensor the curvature of $\mathcal{X}$. Similarly, we can glue all the $\tilde{\mathrm{ric}}_i$ to get an orbi-tensor $\mathrm{ric}_g$ 
on $\mathcal{X}$.

\subsection{Metric structures on orbifolds}

Let $\phi:(\tilde{U},H)\rightarrow U$ be a chart on $\mathcal{X}$, with $\tilde{g}$ representing $g$ locally. Let $\tilde{p}\in \tilde{U}$ be a pre-image of the point $p\in U$. If $\tilde{c}:[0,\epsilon) \rightarrow \tilde{U}$ is a geodesic emanating from $\tilde{p}$, as $H$ acts by isometry 
on $\tilde{U}$, we know that $g\cdot \tilde{c}=L_g\circ \tilde{c}$ is a 
geodesic emanating from $g\cdot\tilde{p}$. If $V\in T_{\tilde{p}}\tilde{U}=\tilde{c}'(0)$, then $TL_g(V)=g\cdot V=(L_g\circ \tilde{c})'(0)\in T_{g\tilde{p}}\tilde{U}$. In the orbi-fibre $T_p\mathcal{X}=\mathbb{R}^n/G_p$, $V$ and $g\cdot V$ represent same orbi-vector. Hence taking $c_{[V]}:=\phi\circ \tilde{c}$, we call it the geodesic emanating from $p$ 
determined by the orbi-vector $[V]$. It is obvious that the definition does not depend on the choice of orbi-chart. 

From the construction, we also note that for $v\in T_p\mathcal{X}$, the 
geodesic $c_v:I\rightarrow \mathcal{X}$ is smooth.

\begin{definition}
Let $p\in X$ be a point. Take $O:=\{v\in T_p\mathcal{X}:c_v$ is defined 
on $[0,1]\}$. We define the exponential map $\exp_p:O\rightarrow X$ by $\exp(v)=c_v(1)$.
\end{definition}

As a topological map, $\exp_p$ is continous. Note that $[0]\in O$ and $\exp_p{[0]}=p$. If $\tilde{U}_p,G_p$ is a fundamental chart at $p$, then $\exp_p$ 
has as lifting $\exp_{\tilde{p}}:\tilde{\Omega}\rightarrow \tilde{U}_p$, where $\tilde{\Omega}\subset T_{\tilde{p}}\tilde{U}$ is an $G_p$-invariant 
open subset containing $0$, and $exp_{\tilde{p}}$ is the classical Riemannian exponential map. One can see that $O$ is indeed of the form $\tilde{O}/G_p$ for some $G_p$-invariant open subset of $\mathbb{R}^n$, hence should be regarded as an open sub-orbifold of $T_p\mathcal{X}$. As $\exp_{\tilde{p}}:\tilde{\Omega}\rightarrow \tilde{U}_p$ gives $G_p$-invariant local diffeomorphism at $0$. We know that $\exp_p$ restricts to some $W=[\tilde{W}/G_p]$ gives an open embedding.

We also remark that if $p\in X_{\mathrm{reg}}$, then the geodesics, hence the exponential map on $X$ are identical to the Riemannian ones around $p$.

\begin{rem}
With the Levi-Citiva connection defined on $(\mathcal{X},g)$, one may consider define a covariant connection along a smooth curve $c:I\rightarrow \mathcal{X}$. However, even when $c$ is lifted as $\tilde{c}:I\rightarrow \tilde{U}$, we don't know if all the lifts are of the form $g\cdot \tilde{c}$. Another hurdle for mere smooth curves is that the definition of orbi-vector fields along them. One of the possible definition is to restrict 
the curves to be strong curve, i.e. $c:I\rightarrow \mathcal{X}$ is strong. In this situation, we could pull the orbi-vector bundle $T\mathcal{X}$ 
together with $\nabla$ back on $I$ via $c$. If the strong curve $c$ has image in $X_\mathrm{reg}$, the definition coincides with the classical one.
\end{rem}

We follow the treatment of \citep{Bor93} for the metric aspects of Riemannian orbifolds. For the basics of metric spaces, we refer the reader to Chapter 1 of \cite{Hae99}. 

For $(\tilde{U},G)$ a chart for $\mathcal{X}$, let $\tilde{g}$ be the representation of $g$ over $\tilde{U}$. Then $(\tilde{U},d_{\tilde{g}})$ is well-defined metric space. If $\lambda:\tilde{V}\rightarrow\tilde{U}$ is an orbifold embedding, then $\lambda$ is an isometry. If for a continuous 
curve $c:I\rightarrow X$, we have local lifts on charts that cover $c(I)$, we may then define the length of $c$ by adding the lengths of its local 
liftings. We now precise the definition.

First, we have
\begin{thm}
\label{lift}
Let $X$ be a left $G$-space, with $G$ a compact Lie group. Let $f:I\rightarrow X/G$ be any path. Then there exists a lifting $f':I\rightarrow X$ covering $f$, i.e. $p\circ f'=f$.
\end{thm}
\begin{proof}
See \cite[Chapter~2, Lemma~6.1]{Bredon}.
\end{proof}

By a compactness argument, we see for a path $c:[0,1]\rightarrow X$, there exists a partition $0=t_0<t_1<...<t_k=1$, orbifold charts $(\tilde{U}_i,G_i), 1\leq i\leq k$ such that $c|_{[t_{i-1},t_{i}]}$ has image in $U_i$ and lifting $\tilde{c}_i$ in $\tilde{U}_i$.

As the liftings are not unique, we give the following definition
\begin{definition}
Let $(\mathcal{X},g)$ be a Riemannian orbifold with underlying space $X=|\mathcal{X}|$. Let $c:[0,1]\rightarrow X$ be a path. Let $\mathcal{P}$ be the set of the local liftings that glue back to $c$, i.e. an element of $\mathcal{P}$ is a triple $(A,B,C)$ where $A$ is a partition $0=t_0<t_1<...<t_k=1$, $B$ is a family of chart $(\tilde{U}_i,G_i), 1\leq i\leq k$ and $C$ is a family of curves $\tilde{c}_i:[t_{i-1},t_i]\rightarrow \tilde{U}_i$ that cover $c|_{[t_{i-1},t_{i}]}$. We define the length of $c$ to be 
\begin{center}
$\mathrm{L}_g(c)=\inf_\mathcal{P} \sum \mathrm{L}_i(\tilde{c}_i)$.
\end{center}
\end{definition}

If $c_1:[0,1]\rightarrow X$, $c_2:[0,1]\rightarrow X$ are two curves such 
that $c_1(1)=c_2(0)$, then we may consider the curve $c_1\ast c_2:[0,1]\rightarrow X$ defined by $t\in [0,\frac{1}{2}] \mapsto c_1(2t)$ and $t\in[\frac{1}{2},1]\mapsto c_2(2t-1)$. And easy observation is that $\mathrm{L}_g(c_1\ast c_2)=\mathrm{L}_g(c_1)+\mathrm{L}_g(c_2)$. Hence the definition is coherent with the intuition of the length of a curve.

\begin{lemma}
\label{liftingtoauniquechart}
Let $c:[0,1]\rightarrow \tilde{U}/G $, $\tilde{x}$ be a pre-image of $x=c(0)$. If there exists a partition $0=t_0<t_1<...<t_k=1$, orbifold charts $(\tilde{U}_i,G_i), 1\leq i\leq k$ such that $c|_{[t_{i-1},t_{i}]}$ has image in $U_i$ and lifting $\tilde{c}_i$ in $\tilde{U}_i$, then there is a lifting $\tilde{c}:I\rightarrow \tilde{U}$ such that $\mathrm{L}(\tilde{c})=\sum \mathrm{L}_i(\tilde{c}_i)$.
\end{lemma}
\begin{proof}
Suppose we have constructed $\tilde{c}:[0,t_{i}]\rightarrow \tilde{U}$ that lift $c|_{[0,t_i]}$ with length equals to $\sum_{j\leq i}\mathrm{L}_i(\tilde{c}_i)$. We now extend $\tilde{c}$ on $[t_i,t_{i+1}]$. We have the projection of $\tilde{c}(t_i)$ and $\tilde{c}_i(t_i)$ on $X$ equal $c(t_i)$. Hence by 
the definition of orbifolds, there exists a chart $(\tilde{V},\tilde{y})$ 
at $c(t_i)$ and two chart embeddings $\lambda:\tilde{V}\rightarrow \tilde{U}$ and $\rho:\tilde{V}\rightarrow \tilde{U}_i$ such that $\lambda(\tilde{y})=\tilde{c}(t_i)$ and $\rho(\tilde{y})=\tilde{c}_i(t_i)$. If $\rho(\tilde{V})$ contains $\tilde{c}_i([t_i,t_i+\epsilon])$,we may then extend $\tilde{c}$ on $[t_i,t_i+\epsilon]$ via $\lambda\circ (\rho)^{-1}\circ 
\tilde{c}_i$. Note $\rho$ and $\lambda$ are isometries, we have $\mathrm{L}(\tilde{c}|_{[t_i,t_i+\epsilon]})=\mathrm{L_i}(\tilde{c}|_{[t_i,t_i+\epsilon]})$. A compactness argument shows we can construct $\tilde{c}$ on $[t_i,t_{i+1}]$ 
with $\mathrm{L}(\tilde{c}|_{[t_i,t_{i+1}]})=\mathrm{L_i}(\tilde{c}_i)$.
\end{proof}

Hence for a curve $c:I\rightarrow \tilde{U}/G$, we may define its length by only considering its liftings on $\tilde{U}$.

From the length $\mathrm{L}_g$, we can define the distance $\mathrm{d}_g$ on $X$ by 
\begin{center}
$ \mathrm{d_g}(x,y)=\inf \mathrm{L_g}(\gamma)$
\end{center}
where the infimum is taken over all the curves that join $x$ and $y$, with the convention $\inf_{\emptyset}=\infty$.

We see easily that $\mathrm{d}_g$ is a distance on $X$ and $\mathrm{d}(x,y)=\infty$ iff $x$ and $y$ are in different components of $X$.

Let $(\tilde{U}_x,G_x)$ be a fundamental chart at $x\in X$. As $\tilde{U}_x$ is a Riemannian manifold, we know there is a $\delta>0$ such that for 
any $\tilde{y},\tilde{z} \in B_{\delta}(\tilde{x})$ there is a unique geodesic $\tilde{c}$ with endpoints $\tilde{y}$ and $\tilde{z}$, such that $\mathrm{L}(\tilde{c})=\tilde{\mathrm{d}}(\tilde{y},\tilde{z})$. By taking the projection $\tilde{U}\rightarrow U$ and combine \cref{liftingtoauniquechart}, we have

\begin{lemma}
\label{g1}
The metric topology induced by $\mathrm{d}_g$ is the same as the original topology 
on $X$. For any $x\in X$, there exists $\delta>0$ such that for any $y,z\in B_\delta(x)\subset X$, there exists a unique geodesic $c$ with endpoints $y$ and $z$ and $\mathrm{L}_g(c)=\mathrm{d}_g(y,z)$.
\end{lemma}

Note for a metric space $(X,d)$ there is a natural length $\mathrm{L_d}$ associated to a curve $c:I\rightarrow X$.
\begin{definition}
\label{metriclength}
Let $(X,d)$ be a metric space, $c:[a,b]\rightarrow X$ is a curve. The length $\mathrm{L_d}(c)$ of $c$ is defined by
\begin{center}
$\mathrm{L_d}(c)=\sup_{a=t_0\leq t_1\leq ...\leq t_n=b} \sum_i d(c(t_i),c(t_{i+1}))$,
\end{center}
where the supremum is taken over all possible partitions (no bound on $n$) with $a=t_0\leq t_1\leq ...\leq t_n=b$.
\end{definition}

A direct consequence of \cref{g1} is:

\begin{cor}
\label{g2}
Let $c:I\rightarrow \mathcal{X}$ be a geodesic. Then with respect to the metric $d_g$, $c$ is locally minimising, i.e. $\forall t\in I, \exists \epsilon > 0$ such that $\forall t_1,t_2\in (t-\epsilon,t+\epsilon)$ one has $\mathrm{L_d}(c|_{[t_1,t_2]})=\mathrm{d}_g(c(t_1),c(t_2))$.
\end{cor}

Suppose $c:[0,1]\rightarrow X$ be a geodesic. By compactness, using \cref{g2} and \cref{g1}, we have an $\epsilon>0$ such that for $s,t\in [0,1], |s-t|<\epsilon$, 
\[
\mathrm{d}_g(c(s),c(t))=\mathrm{L_d}(c|_{[s,t]})\leq \mathrm{L}_g(c|_{[s,t]})=\mathrm{d_g}(c(s),c(t)).
\]
Hence 
\[
\mathrm{L}_g(c)=\sum \mathrm{L}_g(c|_{[t_i,t_{i+1}]})=\sum \mathrm{L_d}(c|_{[t_i,t_{i+1}]})=\mathrm{L_d}(c).
\]
We have thus

\begin{prop}
For a geodesic $c:[0,1]\rightarrow X$, we have $\mathrm{L}_g(c)=\mathrm{L_d}(c)$.
\end{prop}

\begin{prop}[\emph{cf.} \protect{\citep[Page~6]{Bor93}}]
Let $(\mathcal{X},g)$ be a Riemannian orbifold with underline space $X=|\mathcal{X}|$. With the metric $\mathrm{d}_g$ defined above, $(X,\mathrm{d}_g)$ is a length space.
\end{prop}
\begin{proof}
Let $\mathrm{d_i}$ denote the inner metric associated with $\mathrm{d}_g$. We only need to show $\mathrm{d}_g\geq \mathrm{d_i}$. Suppose $\mathrm{d}_g(x,y)<\infty$. Take $\epsilon>0$. By definition, there exists a curve $\gamma:[0,1]\rightarrow X$ such that $\gamma(0)=x, 
\gamma(1)=y$ and $\mathrm{L}_g(\gamma)\leq \mathrm{d}_g(x,y)+\epsilon$. By \cref{g1}, there exists $\delta$ such that if $|t-t'|<\delta$, there exits a geodesic $c$ with endpoints $\gamma(t),\gamma(t')$ satisfying $\mathrm{d}_g(\gamma(t),\gamma(t'))=\mathrm{L}_g(c)$. Let $0=t_0<t_1<...<t_n=1$ be a partition of $[0,1]$ 
with $t_{i+1}-t_i<\delta$ and $c_i$ the corresponding geodesic. Let $c=\ast_i c_i$. We have
\begin{align*}
\mathrm{d}_g(x,y)+\epsilon\geq \mathrm{L_g}(\gamma)=\sum \mathrm{L_g}(\gamma|_{[t_i,t_{i+1}]})&\geq \sum \mathrm{d_g}([t_i,t_{i+1}])=\sum \mathrm{L_g}(c_i)\\
&\geq \sum \mathrm{L_d}(c_i)=\mathrm{L_d}(c)\geq \mathrm{d_i}(x,y).
\end{align*}
\end{proof}

We now recall the
\begin{thm}[Hopf-Rinow]
Let X be a length space. If X is complete and locally compact, then 
\begin{enumerate}
\item every closed bounded subset of X is compact;
\item  any two point of $X$ can be joined by a segment.
\end{enumerate}
\end{thm}
\noindent For the proof, see \cite[Proposition~3.7]{Hae99}.

\medskip
As any segment is locally minimizing, combining this fact with \cref{g2}, we have: 
\begin{lemma}
Let $(\mathcal{X},g)$ be a Riemannian orbifold. If $(X,\mathrm{d}_g)$ is complete, 
then any two point can be joined by a minimising geodesic.
\end{lemma}

\begin{rem}
In the case of a Riemannian manifold $(M,g)$, one associates a metric $d_g$ on $M$ by defining $d_g(x,y)=\inf \mathrm{L}_g(c)$, where $c$ is a piece-wise smooth curve and $\mathrm{L}_g(c)=\int |c'(t)|dt$. For the metric $d_g$, we can associate another length $\mathrm{L}_d$ as in \cref{metriclength}. One sees easily $\mathrm{L}_g\geq \mathrm{L}_d$ and it is a classical result that $\mathrm{L}_g=\mathrm{L}_d$. The author does not know if this still holds in the orbifold setting.
\end{rem}

For a complete orbifold $(\mathcal{X},g)$, its geodesics have a good property:

\begin{thm}[\emph{cf.} \protect{\cite[Lemma~13]{Bor93}}]
\label{geo}
Suppose $\gamma :I=[0,1]\rightarrow \mathcal{X}$ is a minimal geodesic. 
Let $p=\gamma(0),q=\gamma(1)$. Then we have one of the following mutually exclusive conditions:
\begin{enumerate}
\item $\gamma(I)\subset X_\mathrm{sing}$
\item $\gamma(I)\cap X_\mathrm{sing}\subset \{p,q\}$.
\end{enumerate}
\end{thm}
Hence for $p,q\in X_\mathrm{reg}$, the minimal geodesic $\gamma$ joining $p$ and 
$q$ lies completely in $X_\mathrm{reg}$. One sees then $\gamma$ is also the minimal geodesic in the Riemannian manifold $X_\mathrm{reg}$.

\begin{cor}
\label{Convex}
Let $(\mathcal{X},g)$ be a Riemannian orbifold. If $(X,d_g)$ is complete, 
then $(X_\mathrm{reg},g)$ is a convex Riemannian manifold.
\end{cor}

Let $(M,g)$ be a convex Riemannian manifold, $p \in M$. For any $u\in U_pM$ a unit tangent vector, $\gamma_u$ the geodesic emanating from $p$ with 
$\gamma_u'(0)=u$, we define
\[
t(u):=\sup \{t > 0\mid \gamma_u\ \text{is defined on}\ [0,t]\ \text{and}\ \gamma_u|_{[0,t]}\ \text{is minimal}\in [0,\infty] \}.
\]

We may define the cut locus in the convex situation by 
\begin{definition}
Let $M$ be a convex Riemannian manifold and $p\in M$. We define
\[\tilde{C}_p:=\{t(u)u\mid u\in U_pM\ \text{such that} \ t(u)<\infty\}\quad\text{and} \quad C_p:=\exp_p(\tilde{C}_p)\]
the tangent cut locus and cut locus of $p$ respectively. We also set
\[\tilde{I}_p:=\{tu\mid u\in U_pM,\ 0 < t <t(u) \}\quad\text{and}\quad I_p:=\exp_p(\tilde{I}_p).\]
\end{definition}
Then we have similar results as in the complete case:
\begin{lemma}
\leavevmode
\begin{enumerate}
\item $I_p\cap C_p=\emptyset$, $M=I_p\cup C_p$, and $\bar{I_p}=M$.
\item $\exp_p:\tilde{I}_p\rightarrow I_p$ is a diffeomorphism.
\item $C_p$ has measure 0.
\end{enumerate}
\end{lemma}
\begin{proof}
See \cite[Proposition~III.4.1 and Lemma~III.4.4]{Sakai96}. Note though all the statements are for complete manifold $M$, the proofs only use the fact that $M$ is convex.
\end{proof}

\subsection{Volume comparisons}

We first recall the classical Bishop-Gromov comparison theorem:
\begin{thm}
\label{B-G}
Let $(M^n,g)$ be a convex Riemannian manifold with $\ric_g\geq (n-1)k$. Let $v(n,k,r)$ be the volume of a ball of radius $r$ in the model space with constant curvature $k$. The volume ratio
\begin{equation*}
r\mapsto\frac{\vol_g(B(p,r))}{v(n,k,r)}
\end{equation*}
is a non-increasing function whose limit is $1$ as $r \to 0$.
\end{thm}
For the proof, see Lemma 7.1.4 in \citep{Petersen16}. Note the proof only uses the fact $\exp_p:\tilde{I}_p\rightarrow I_p$ is a diffeomorphism and that $C_p$ has measure zero.

Before we state the orbifold version Bishop-Gromov theorem,  we first need to define the measure on $X$ for a Riemannian orbifold $(\mathcal{X},g)$.

 For $(\mathcal{X},g)$, if $(\tilde{U},\tilde{g})$ is a chart, then after 
taking an orientation of $\tilde{U}$, we have a unique volume form $\vol_{\tilde{g}}$. Hence $\vol_{\tilde{g}}$ is defined up to a sign. If $O\subset X$ is an open subset covered by $\tilde{U}_i/G_i$, we can take a partition of 
unity $\rho_i$ with respects to $\{U_i\}$ and define
\[\vol_g(O):=\sum_i |\int _{\mathcal{X}} \rho_i \vol_{\tilde{g_i}}|.\]
It's easy to see $\vol_g(O)$ is well defined.
 
\begin{lemma}[\emph{cf.} \protect{\cite[Lemma~18]{Bor93}}]
Let $(\mathcal{X},g)$ be a Riemmanian orbifold and $\nu$ be its canonical measure on $X$. Then $X_\mathrm{sing}$ has measure zero.
\end{lemma} 
\begin{proof}
Note $X_\mathrm{sing}$ is a closed subset of $X$, hence it is measurable. As $X$ 
is second countable, we may cover $X_\mathrm{sing}$ by countable many local charts. Hence it suffices to show for $(\tilde{U},G)$ the non-free point set has measure zero. As $G$ is finite, it is trivial.
\end{proof}

Now we state the orbifold Bishop-Gromov theorem
\begin{thm}[\emph{cf.} \protect{\cite[Proposition~20]{Bor93}}]
\label{Gromov-Bishop}
Let $(\mathcal{X},g)$ be a Riemannian orbifold with $\ric_g\geq (n-1)k$. Let 
$v(n,k,r)$ be the volume of a ball of radius $r$ in the model space with constant curvature $k$. The volume ratio
\begin{equation*}
r\mapsto\frac{\vol_g(B(p,r))}{v(n,k,r)}
\end{equation*}
is a non-increasing function whose limit is $|G_p|$ as $r \to 0$.
\end{thm}
\begin{proof}
First suppose $p\in X_\mathrm{reg}$. We have $\vol_g(B(p,r))=\vol_g(B(p,r)\cap X_\mathrm{reg})$. Let $B'(p,r)=B(p,r)\cap X_\mathrm{reg}$ and $\vol'_g$ be the volume on $X_\mathrm{reg}$. We have
\[\frac{\vol_g(B(p,r))}{v(n,k,r)}=\frac{\vol'_g(B'(p,r))}{v(n,k,r)}\]
hence by \cref{B-G}, we get the sought result.

If $p\in X_\mathrm{sing}$, let us pick $(p_i)_{i\ge 1}\in X_\mathrm{reg}$ a sequence that converges to $p$. Then $\vol_g(B(p,r))=\lim\limits_{i\to \infty}\vol_g(B(p_i,r))$. For any $r'>r$, by the results for regular points, we have
\[\frac{\vol_g(B(p_i,r))}{\vol_g(B(p_i,r')}\geq \frac{v(n,k,r)}{v(n,k,r')}.\]
After taking the limit as $i\to \infty$, we have
\[\frac{\vol_g(B(p,r))}{\vol_g(B(p,r'))}\geq \frac{v(n,k,r)}{v(n,k,r')}.\]
For $r$ small enough, we have $B(p,r)\subset \tilde{U}_p/G_p$, with $(\tilde{U}_p,G_p)$ a fundamental chart at $p$. If we denote by $\tilde{B}(\tilde{p},r)$ the ball of radius $r$ in $\tilde{U}_p$ centred at $\tilde{p}$, we then have:
\[\vol_{\tilde{g}}\tilde{B}(\tilde{p},r)=|G_p|\cdot \vol_g(B(p,r)).\]
The limit $|G_p|$ as $r\to 0$ follows from \cref{B-G}.
\end{proof}


\section{Orbifold coverings and generalized Magulis lemma}
\label{sec4}
\subsection{Orbifold coverings and orbifold fundamental group}
There are several different but equivalent definitions of orbifold covering spaces and the respective definition of orbifold fundamental groups. These groups are however canonically isomorphic. We will treat Thurston's original definition in this section and the pair definition in \cref{sec6}. For the covering groupoid definition, we refer to \citep[Chapter~2.2.]{Ruan07}. For the $\mathcal{G}$-path definition, we refer to \citep[Chapter~III.\protect{$\mathcal{G}$}.3.]{Hae99}.

We begin by recalling the definition of covering map.
\begin{definition}[[\emph{cf.} \protect{\cite[Definition~13.2.2.]{Thurston79}}]
A smooth map $p:\mathcal{X}\rightarrow \mathcal{Y}$ is called an orbifold 
covering map, if for any $y\in Y$, there exists an orbifold chart $(\tilde{V},G,\psi)$ such that each component $U_i$ of $p^{-1}(V)$ has $(\tilde{V},G_i,\phi_i)$ as a chart for some subgroup $G_i$ of $G$ and $p$ is lifted with respect to this chart as identity.
We call such a neighborhood $V$ an elementary neighborhood with respect to $p$ at $y$.
\end{definition}

We note that the name elementary neighborhoods are not a standard name in literature. Quite surprisingly, there are neighborhoods of a point $y\in Y$ that are always elementary neighborhoods with respect to any covering $p$. More precisely:
\begin{prop}
Let $p:\mathcal{X}\rightarrow \mathcal{Y}$ be a covering map and $p\in Y$. Let $(\tilde{V},G)$ be a chart at $y$. If $\tilde{V}$ is simply connected then $V=\tilde{V}/G$ is an elementary neighbourhood with respect to $y$.
\end{prop}
\begin{proof}
See \cite[Proposition 6]{Choi04} .
\end{proof}

We will only use a direct consequence of the proposition in this paper.
\begin{cor}
Let $p:\mathcal{X}\rightarrow \mathcal{Y}$ be a covering map and $p\in Y$. The elementary neighborhoods at $y$ with respect to $p$ form a basis of the topology of $Y$ at $y$.
\end{cor}

\begin{definition}
Let $\mathcal{X}$ be a connected orbifold, $x\in X_\mathrm{reg}$. We call a covering map $p:(\tilde{\mathcal{X}},\tilde{x})\rightarrow (\mathcal{X},x)$ a 
universal covering if $|\tilde{\mathcal{X}}|$ is connected and for any covering $q:(\mathcal{Y},y)\rightarrow (\mathcal{X},x)$ there exists a unique smooth map $f:(\tilde{\mathcal{X}},\tilde{x})\rightarrow (\tilde{Y},y)$ such that $p=q\circ f$.
\end{definition}

\begin{thm}[Thurston]
For any orbifold $\mathcal{X}$, its universal cover exists and is unique up to isomorphism.
\end{thm}
\begin{proof}
see \cite[Proposition 13.2.4.]{Thurston79} or \cite[Proposition 8]{Choi04}.
\end{proof}

\begin{definition}
\label{pi1orb}
Let $p:\tilde{\mathcal{X}}\rightarrow \mathcal{X}$ be the universal covering of $\mathcal{X}$. The orbifold fundamental group $\pi_1^{\mathrm{orb}}(\mathcal{X})$ is defined to be $\mathrm{Aut}(p)$. 
\end{definition}

\begin{prop}
Let $p:\tilde{\mathcal{X}}\rightarrow \mathcal{X}$ be the universal covering of $\mathcal{X}$. The orbifold fundamental group $\pi_1^{\mathrm{orb}}(\mathcal{X})$ acts by isomorphism on $\tilde{\mathcal{X}}$. For any $x\in X_\mathrm{reg}$, $F_x:=p^{-1}(x)\subset \tilde{X}_\mathrm{reg}$, then $\pi_1^{\mathrm{orb}}(\mathcal{X})$ acts freely and transitively on $F_x$
\end{prop}
\begin{proof}
See \citep[Corollary~2.(i)]{Choi04}
\end{proof}

We will need the following theorem in the sequel.
\begin{thm}[Thurston]
\label{Galois}
Let $\mathcal{X}$ be an orbifold. Then there is a one-one correspondence between the isomorphism classes of orbifold coverings of $\mathcal{X}$ and the conjugacy classes of subgroups of $\pi_1^{\mathrm{orb}}(\mathcal{X})$.
\end{thm}
\begin{proof}
See \citep[Corollary~2.(2)]{Choi04}
\end{proof}

Suppose that $p:\mathcal{Y} \rightarrow \mathcal{X}$ is a covering map. One could show that $p$ is indeed strong. Hence if $g$ is a metric on $X$, we could consider $p^{\ast}(g)$ on $\mathcal{Y}$. In particular, for $(\mathcal{X},g)$, $p$ induces a morphism of metric spaces $(Y,d_Y)\rightarrow (X,d_X)$ such that $d_Y(a,b)\geq d_X(p(a),p(b))$.

For a Galois covering $p:\mathcal{Y}\rightarrow \mathcal{X}$, $d_X$ and $d_Y$ are nicely related:

\begin{prop}\label{isom}
Let $p$ be as above and $\mathrm{Gal}(p)$ be the Galois group of $p$. Then $X=Y/\mathrm{Gal}(p)$ and $d_X$ is the quotient metric of $d_Y$ by $\mathrm{Gal}(p)$.
\end{prop}
\noindent For the proof, see \cite[Lemma 2.8]{Lange20}.

One sees from this if $d_Y=d_X$, then $p:X\rightarrow Y$ is a covering map in the category of topological spaces.
 
\medskip
For our purpose, we state the following lemma.
\begin{lemma}
Let $(\mathcal{X},g)$ be a Riemannian orbifold and $p:{\mathcal{Y}}\rightarrow \mathcal{X}$ a covering of $\mathcal{X}$. If $(X,d_X)$ is complete, 
$(Y,d_Y)$ is complete.
\end{lemma}
\begin{proof}
Let $\{y_n\}$ be a Cauchy sequence in $(Y,d_Y)$. We may assume that $\{y_n\}$ lies in a compact subset $K$ of a component $V$ of the pre-image $p^{-1}(U)$ of an elementary neighborhood $U=[\tilde{U}/G]$. Hence there exists $H\leqslant G$ such that $V=[\tilde{U}/H]$. Let $U\ni x_n=p(y_n)$ be the image of $y_n$ under $p$. We know that $x_n\in p(K)$ and as $p$ does not increase distances, $\{x_n\}$ is a Cauchy sequence with limit $\lim\limits_{n\to\infty}x_n=x\in p(K) \subset U$. Take $\tilde{x}\in \tilde{U}$ be a lifting of $x$. Then there exist $\tilde{x}_n$ lifting $x_n$ such that $\tilde{x}_n \to \tilde{x}$. Note for each $n$ there exists $g_n \in G$ such that $g_n \tilde{x}_n$ lifts $y_n$. As $G$ is finite, there exists $g\in G$ that appears infinitely many times in $\{g_n\}$. Hence there is a sub-sequence $g\cdot \tilde{x}_{n_i}\to g\cdot \tilde{x}$. If $y$ is the 
image of $g\cdot\tilde{x}$ in $V$, we finally have $y_n\to y$.
\end{proof}

By \cref{lift}, we see that for any covering $p:\mathcal{X}'\rightarrow \mathcal{X}$ and path $c:[0,1] \rightarrow \mathcal{X}$ with $c(0)=x\in X$ and $x'\in p^{-1}(x),$ there exists lifting $c':[0,1]\rightarrow \mathcal{X}'$ starts from $x'$. However, this will not be unique. Consider $(\mathbb{R}^2,\mathbb{Z}/2\mathbb{Z})$ with the action $(x,y)\rightarrow (-x,y)$. The path $t\rightarrow (t,0)$ has lifting $t\rightarrow (t,0)$ and $t\rightarrow (-t,0)$ in $\mathbb{R}^2$. If $(\tilde{U},G_{x'})\rightarrow (\tilde{U},G_x)$ is a local lifting of $p$, and $c$ is a geodesic, then it has at least $[G_x:G_{x'}]$ lifting. 

\subsection{Dirichlet domains and generalized Margulis lemma}
Let $\pi:\mathcal{X}'\rightarrow\mathcal{X}$ be the universal covering of 
$(\mathcal{X},g)$ and $\Gamma$ be its Galois group. We endow $\mathcal{X}'$ with 
the pullback metric $p^{\ast}(g)$. Then $(X',d')$ is a complete length space and $\Gamma$ acts on $X'$ such that $d'(\gamma \cdot y,\gamma\cdot z)\leq 
d'(y,z)$. Hence $\Gamma$ acts by isometries on $X'$. Suppose now we have $\diam(X)\leq 1$. Let $x_0\in \pi^{-1}(X_\mathrm{reg})$. We now show some basic properties of the Dirichlet domain of $\Gamma$ based at $x_0$:
\begin{center}
$F:=\{p\in X':\mathrm{d'}(x_0,p)\leq \mathrm{d'}(\gamma\cdot x_0,p)$ for all $\gamma\in\Gamma\}$.
\end{center}

\begin{lemma}
$\diam(F)\leq 2$.
\end{lemma}
\begin{proof}
Suppose $p\in F$. As $X'$ is complete, there is a segment $c':[0,1]\rightarrow X'$ joining $x_0$ and $p$ such that $\mathrm{L}_d(c')=\mathrm{d'}(x_0,p)$. In particular, $c'$ is a geodesic. We consider $c:=\pi \circ c'$, which joins $\pi(x_0)$ and $\pi(p)$. If $c$ is not a segment, then there is a $c_1$ which is a segment that joins $\pi(x_0)$ and $\pi(p)$. We have $\mathrm{d}_g(\pi(x_0),\pi(p))<\mathrm{L_d}(c)=\mathrm{L}_g(c)$. As $c_1$ is a geodesic, it has a lifting $c'_1$ starting at $x_0$. As $\pi\circ c'_1=c_1\neq c=\pi\circ c'$, we have $c'_1\neq c'$. Let $\gamma\cdot p=c'_1(1)$. Then 
\begin{equation*}
\mathrm{d'}(x_0,\gamma p)\leq \mathrm{L_d}(c'_1)=\mathrm{L_d}(c_1)<\mathrm{L_d}(c)=\mathrm{d'}(x_0,p). 
\end{equation*}
Thus $c$ is also a segment and $\mathrm{L}(c')=\mathrm{L}(c)\leq 1$. Thus $\diam(F)\leq 2$.
\end{proof}

We consider the subset of $\Gamma$ defined by 
\begin{center}
$S:=\{\gamma\in \Gamma:\mathrm{d}(\gamma\cdot x_0,x_0)\leq 3\}$.
\end{center}

One sees easily that $S$ is symmetric and contains $1$. We have
\begin{lemma}
\label{l1}
$S$ generates $\Gamma$.
\end{lemma}
\begin{proof}
One has $\cup\gamma\cdot F=X'$. We take a segment $c:[0,1]\rightarrow X'$ that joins $x_0,\gamma_0\cdot x_0$. As $c(I)$ is compact, $c(I)$ is contained in a ball $B_r(x_0)$ that meets only finitely many translates $\gamma\cdot F$ of $F$. Hence $c$ passes through finitely many $\gamma\cdot F$. We list these elements by $1=\gamma_1,\ldots,\gamma_k=\gamma_0$, ordered by the time when $c$ enters $\gamma_i$, we note they are not necessarily different. Then $\gamma_i\cdot F\cap \gamma_{i+1}\cdot F\neq \emptyset$ and we thus have $\mathrm{d'}(\gamma_i x_0,\gamma_{i+1} x_0)\leq 2$. Finally we remark that
\[\gamma_0=\gamma_k=\prod_{i=1}^{k-1}
(\gamma_{i+1}\cdot\gamma_i^{-1})\in S^{k-1}.\]
\end{proof}

\begin{lemma}
\label{l2}
Let $r>0$ be an integer. $B(x_0,r)\subset S^r\cdot F\subset B(x_0,3r+2)$.

\end{lemma}
\begin{proof}
If $p\in S^r\cdot F$, $p=\gamma_1...\gamma_r\cdot q$ with $q\in F$ and $\gamma_i\in S$. It yields
\[d'(x_0,p)\leq d'(\gamma_1\cdots\gamma_r\cdot x_0,x_0)+d'(\gamma_1\cdots\gamma_r\cdot x_0,\gamma_1\cdots\gamma_r\cdot q)\leq 2+3r.\]
If $p\in B(x_0,r)$, as $\cup\gamma\cdot F=X'$, there exists $\gamma_0\in \Gamma$ such that $\gamma_0^{-1}\cdot p\in F$. Then $d'(x_0,\gamma\cdot x_0)\leq d'(x_0,p)+d'(\gamma_0\cdot x_0,p)\leq r+1$. Hence $\gamma_0\in S^r$.
\end{proof}

For any $\gamma\neq 1$, $\gamma \cdot F\cap F\subset \partial F$ and we have $\partial F=\cup_{\gamma\neq 1}(\gamma \cdot F\cap F)$.
\begin{lemma}
\label{l3}
Let $\mu$ be the canonical measure associated with $p^{\ast}(g)$. Then $\mu(\partial F)=0$.
\end{lemma}
\begin{proof}
Consider a point $p\in \partial F$. If $\pi(p)\in X_\mathrm{reg}$,  one has $p\in X'_\mathrm{reg}$. Suppose $p\in F\cap \gamma\cdot F$. Then $d'(x_0,p)=d(\gamma\cdot x_0,p)$. The proof of \cref{l1} shows that we could then find two 
distinct segments on $X$ joining $\pi(x_0)$ and $p$. By \cref{geo}, these two segments are minimizing geodesics in the convex manifold $X_\mathrm{reg}$. Hence $\pi(p)$ lies in the cut locus $C_{\pi(x_0)}$ of $x_0$ in $X_{reg}$. Thus one has $\pi(\partial F)\subset C_{\pi(x_0)}\cup X_\mathrm{sing}$. Now cover $\pi(\partial F)$ by elementary neighborhoods with respect to $\pi$. By second countability, we may find countably many neighborhoods $(\tilde{U}_i,G_i,\phi_i)$. Note $\phi_i^{-1}(C_{\pi(x_0)}\cup X_\mathrm{sing})$ has measure $0$ in $\tilde{U}_i$ and $\partial F$ is covered by countably many $\tilde{U}_i/H_{ij}$. Hence $\mu(\partial F)=0$.
\end{proof}

With \cref{Gromov-Bishop} and \cref{isom} , the Margulis lemma for fundamental groups of compact manifolds with Ricci curvature bounded below by \cite{Tao12} holds for the orbifold case, thus we have
\begin{prop}[\emph{cf.} \protect{\cite[Corollary 11.13]{Tao12}}]
Given $n\in \mathbb{N}$, there is $\epsilon=\epsilon(n)>0$ such that:
for any $n$-dimensional Riemmanian orbifold $(\mathcal{X},g)$ with compact underlying space $|X|$, if $\ric_g\geq -\epsilon$ and $diam(X)\leq 1$, then $\pi_1^{\mathrm{orb}}(\mathcal{X})$ is virtually nilpotent.
\end{prop}
\begin{proof}
Let $\pi,x_0,F, $ be defined as above. With \cref{l2} and \cref{l3}, we have 

\begin{equation*}
\frac{|S^r|}{|S|}\leq \frac{\mu(B(x_0,3r+2))}{\mu(B(x_0,1))}.
\end{equation*}

By \cref{Gromov-Bishop}, we have 

\begin{equation*}
\frac{\mu(B(x_0,r))}{\mu(B(x_0,1))}\leq \frac{v(n,-\epsilon ,r)}{v(n,-\epsilon,1)}.
\end{equation*}

Let $\omega_n$ be the volume of $(n-1)$-dimensional unit sphere in Euclidean space $\mathbb{R}^n$. Then $v(n,-\epsilon,r)=\omega_n\int_0^r(\frac{sinh(\sqrt{\epsilon} t)}{\sqrt{\epsilon}})^{n-1}dt$. The latter tends to $\omega_n r^n/n$ when $\epsilon$ tends to $0$. Thus for any $R_0\geq 1$, there exists $\epsilon_0=\epsilon_0(d,R_0)$ such that 
\begin{equation*}
\frac{|S^r|}{|S|}\leq 2(3r+2)^n
\end{equation*}
for all $r\leq R_0$ provided that $0<\epsilon<\epsilon_0$. The existence of $\epsilon=\epsilon(d)$ follows from Corollary 11.5 of \cite{Tao12}.
\end{proof}

For the main theorem, we introduce the following notion.
\begin{definition}
Let $(X,d)$ be a metric space. We say $X$ to have bounded packing with packing constant $K$ if there exists $K>0$ such that every ball of radius $4$ in $X$ can be covered by at most $K$ balls of radius $1$.
\end{definition}

\begin{lemma}
\label{pack}
Let $(\mathcal{X},g)$ be a complete Riemannian orbifold with $\ric_g>-(n-1)$. Then $X=|\mathcal{X}|$ has bounded packing with packing constant $K=K(n)$.
\end{lemma}
\begin{proof}
Let $p \in X$. For the ball $B(p,5)$ and a ball $B(q,\frac{1}{2})\subset B(p,5)$, by \cref{Gromov-Bishop} we have
\begin{equation*}
\frac{\vol_g(B(p,5))}{\vol_g(B(q,\frac{1}{2})}\leq \frac{\vol_g(B(q,10))}{\vol_g(B(q,\frac{1}{2})}\leq\frac{v(n,-1,10)}{v(n,-1,\frac{1}{2})}=K(n).
\end{equation*} 
Let $B(q_i,\frac{1}{2})$ be a family of disjoint balls that is contained $B(p,5)$ such that for any $q\neq q_i$, if $B(q,\frac{1}{2})\subset B(p,5)$, then $B(q,\frac{1}{2})$ intersects with one of the $B(q_i,\frac{1}{2})$. We know the family has most $K(n)$ balls. Note the balls $B(q_i,1)$ cover $B(p,4)$. Hence we have the packing constant $K=K(n)$.
\end{proof}

Let $(\mathcal{X},g)$ be a complete Riemannian orbifold and $p:\tilde{\mathcal{X}}\rightarrow \mathcal{X}$ its universal covering. For any point $x\in X=|\mathcal{X}|$, its (topological) fiber $p^{-1}(x)\subset \tilde{X}=|\tilde{\mathcal{X}}|$ is a discrete space. Let us pick $\tilde{x}\in p^{-1}(x)$. We then have $\min\{\tilde{d}(\tilde{x},\tilde{x}')\}>0$, where the minimum is taken for all $\tilde{x}'\in p^{-1}(x)\setminus \{\tilde{x}\}$. We thus see $\pi_1^{orb}(\mathcal{X})$ acts on $\tilde{X}$ discretely, $i.e.$, for any $\tilde{x}\in \tilde{X}$, for any bounded set $\Sigma\subset \tilde{X}$, the set $\{\gamma\in \pi_1^{orb}(\mathcal{X})\mid\gamma\cdot \tilde{x}\in \Sigma\}$ is finite.

Finally we recall the generalized Margulis lemma established in \cite{Tao12}
\begin{thm}[\emph{cf.}\protect{\cite[Corollary 11.17]{Tao12}}]
\label{gm}
Let $K\geq 1$ be a parameter. There exists $\epsilon(K)>0$, such that the following is true. Suppose $X$ is a metric space with packing constant $K$ and $\Gamma$ is a subgroup of isometries of $X$ that acts discretely. Then for every $x\in X$ the "almost stabiliser"
\[\Gamma_\epsilon(x)=\langle\gamma\in\Gamma\mid\mathrm{d}(\gamma\cdot x,x)<\epsilon\rangle\]
is virtually nilpotent.
\end{thm}

With \cref{pack}, applying \cref{gm} to complete Riemannian orbifolds, we get the following lemma.
\begin{lemma}[\emph{cf.} \protect{\cite[Corollary 11.19]{Tao12}}]
\label{mag}
Let $n\geq 1$ be an integer. There exists $\alpha=\alpha(n)>0$ such that the following holds true. Suppose $\mathcal{X}$ is a complete Riemannian orbifold with its Ricci curvature bounded by $\ric\geq -(n-1)$ and $\Gamma$ a subgroup of $Isom(|\mathcal{X}|)$ acting properly discontinuosly by isometries on $|\mathcal{X}|$. Then for every $x\in |\mathcal{X}|$, the "almost stabliser"
\[\Gamma_{\alpha}(x):=\langle\gamma \in \Gamma\mid \mathrm{d}(\gamma\cdot x,x)<\alpha\rangle\]
is vertually nilpotent.
\end{lemma}

\section{main theorem}
\label{sec5}
In this section and onwards, we only deal with complex orbifolds.
\begin{definition}
Let $\mathcal{X}$ be an orbifold and $p:\mathcal{E}\rightarrow \mathcal{X}$ be a complex orbi-vector bundle over $\mathcal{X}$, that is there exists a representation of $\mathcal{X}$ by an orbifold groupoid $\mathcal{G}$ and a left $\mathcal{G}$-space $E$ such that $p$ is represented by $\mathcal{G}\ltimes E\rightarrow \mathcal{G}$. An Hermitian metric on $\mathcal{E}$ is a map $h:|\mathcal{E}|\times_{|\mathcal{X}|}|\mathcal{E}|\rightarrow \mathbb{C}$ such that $h$ lifts to a map $\tilde{h}:E\times_{\mathcal{G}_0}E \rightarrow \mathbb{C}$ and $\tilde{h}$ is Hermitian and $\mathcal{G}$-invariant.
\end{definition}

One can always get a Hermitian metric on an orbi-vector bundle by partition of unity:
\begin{prop}[\emph{cf.} \protect{\cite[Lemma 5.1]{Pardon20vector}}]
Let $\mathcal{X}$ be an orbifold and $\mathcal{E}$ be a complex orbi-vector bundle on $\mathcal{X}$. Then there exists an Hermitian metric on $\mathcal{E}$.
\end{prop}

For a complex orbifold groupoid $\mathcal{G}=\{G_1\rightrightarrows G_0\}$, we know that the structure maps between $G_i$ are holomorphic (actually they are étale). In particular, for any arrow $g:x\rightarrow y$ in $\mathcal{G}_1$, the induced local diffeomorphism $U_x\rightarrow U_y$ is biholomorphic. Its tangent groupoid is $T\mathcal{G}=G_1\ltimes TG_0$. We see that the almost complex structure $J$, differential operators $d$, $\partial$ and $\bar{\partial}$ are $G_1$-invariant. Thus for a complex orbifold $\mathcal{X}$, we have the decomposition $T_{\mathbb{C}} \mathcal{X}=T^{1,0}\mathcal{X}\oplus T^{0,1}\mathcal{X}$. We define the anti-canonical bundle of $\mathcal{X}$ to be $K_{\mathcal{X}}^{-1}=\det(T^{1,0}\mathcal{X})$.

\begin{definition}
Let $\mathcal{X}$ be a complex orbifold. A Kähler form on $\mathcal{X}$ is a closed real $(1,1)$-form $\omega\in \Gamma(X,(T^2\mathcal{X})^{\vee}\cap (T^{(1,1)}\mathcal{X})^{\vee})$ such that $\omega(-,J-)$ defines a Riemannian metric on $\mathcal{X}$.
\end{definition}

We give a definition of holomorphic orbi-vector bundles that suits our later discussion.

\begin{definition}
\label{hol}
Let $\mathcal{X}$ be a complex orbifold, $\mathcal{E}\rightarrow \mathcal{X}$ a complex orbi-vector bundle over $\mathcal{X}$. A holomorphic orbi-vector bundle structure over $\mathcal{E}$ is a representation $\mathcal{G}\ltimes E\rightarrow \mathcal{G}$ of $\mathcal{E}\rightarrow \mathcal{X}$ together with a holomorphic vector bundle structure on $E\rightarrow \mathcal{G}_0$. We call $\mathcal{E}$ together with the holomorphic structure a holomorphic orbi-vector bundle.
\end{definition}

With the above definition, we could see that $(T\mathcal{X},J)$ and $K_{\mathcal{X}}$ carry natural holomorphic structure. It is thus considered as holomorphic orbi-vector bundle in the rest of the article.

\begin{rem}
\label{rhol}
The definition is not optimal. Suppose $\phi:\mathcal{H}\rightarrow \mathcal{G}$ is an equivalence, then one will have a nature holomorphic vector bundle $\phi_0^{\ast}(E)\rightarrow \mathcal{H}_0$. One should consider $\phi^{\ast}(\mathcal{G}\ltimes E)=\mathcal{H}\ltimes \phi_0^{\ast}(E)$ gives same holomorphic structure on $\mathcal{E}\rightarrow \mathcal{X}$. Hence the right definition will be the
\begin{center} $\{$representation $+$ holomorphic structure on the representation$\}$ $mod$ "equivalence".
\end{center} Due to the inability of the author, we could not give a satisfying equivalence relation. However, the given definition suffices for our purpose.
\end{rem}

Let $\mathcal{E}\rightarrow \mathcal{X}$ be a holomorphic orbi-vector bundle represented by $\mathcal{G}\ltimes E\rightarrow \mathcal{G}$. The holomorphic structure on $E$ gives $E$ a natural complex structure and hence $\mathcal{G}\ltimes E$ a complex orbifold groupoid. For $x\in \mathcal{G}_0$, we take $U_x\subset \mathcal{G}_0$ such that $\mathcal{G}|_{U_x}\cong G_x\ltimes U_x$ and $E|_{U_x}\cong \mathbb{C}^r\times U_x$ by a holomorphic frame. Then $(\mathcal{G}\ltimes E)|_{U_x}\cong G_x\ltimes(\mathbb{C}^r\times U_x)$ as complex orbifold groupoid. In particular, $G_x\times (\mathbb{C}^r\times U_x)\rightarrow \mathbb{C}\times U_x$ is holomorphic. For an element $g\in G_x$, the action is  $(v,y)\mapsto (g(y)\cdot v,y)$ where $g(y)\in GL(r,\mathbb{C})$. One sees easily $y\mapsto g(y)$ is holomorphic. Thus $g\in G_x$ transfers holomorphic section of $E$ on $U_x$ to a holomorphic section. With the same argument, one could show $g:x\rightarrow y$ transfers a local holomorphic section around $x$ to a local holomorphic section around $y$. We define the holomorphic section of $\mathcal{E}$ to be a $\mathcal{G}$-invariant holomorphic section in $E$. Let $s$ be a $\mathcal{G}$-invariant holomorphic section of $E$, and $g:x\rightarrow y$ be an arrow in $\mathcal{G}_1$. Suppose that $s$ is defined around $x$ and $y$, such that $s=\sum \phi_i e_i$ around $x$ and $s=\sum \psi_j f_j$ around $y$, where $\{e_i\}$ and $\{f_j\}$ are holomorphic frames. We know $g\cdot e_i=\sum_jA_{ji}f_j$ for some holomorphic functions $A_{ji}$ around $y$ and $A_{ji}$ is invertible. Now $\bar{\partial}_E(s)=\sum_i \bar{\partial}\phi_i\otimes e_i$ around $x$. We have around $y$,
\begin{align*}
g\cdot \bar{\partial}_E(s)&=\sum_i (g^{-1})^{\ast} \bar{\partial}\phi_i\otimes g\cdot e_i=\sum_{i,j}(g^{-1})^{\ast} \bar{\partial}\phi_i A_{ji}\otimes f_j\\
&=\sum_j\bar{\partial} \psi _j\otimes f_j=\bar{\partial}_E(s).
\end{align*}
Thus the Dolbeault operator $\bar{\partial}_E$ passes to an orbifold Dolbeaut operator $\bar{\partial}_{\mathcal{E}}$ on $\mathcal{E}$.

Let $\mathcal{L}\rightarrow \mathcal{X}$ be a holomorphic orbi-line bundle with an Hermitian mectric $h$. We take a representation $\mathcal{G}\ltimes L\rightarrow \mathcal{G}$ of $\mathcal{L}$ and $\tilde{h}$ be the $\mathcal{G}$-invariant metric on $L$. For an arrow $g:x\rightarrow y$, we consider two trivializations by holomorphic sections $e$ and $f$ around $x$ and $y$ respectively. Suppose that $g\cdot e (w)=\phi(g^{-1}w)f(w)$. We have the local matrix for $\tilde{h}$ too be $h_1=\tilde{h}(e,e)$ and $h_2=\tilde{h}(f,f)$. From the equality $h(g\cdot e,g \cdot e)=h(e,e)$, we see that 
\[h_2(w)\phi(g^{-1}w)\overline{\phi(g^{-1}w)}=h_1(g^{-1}w).\]
Hence 
\begin{equation*}
\partial\bar{\partial} (log\circ h_2)=\partial\bar{\partial}(\log \circ h_1\circ L_{g^-1})=(g^{-1})^{\ast}\partial\bar{\partial}(\log \circ h_1)=g\cdot\partial\bar{\partial}(\log \circ h_1),
\end{equation*}
which means that the Chern curvature $\Theta_{\tilde{h}}=-\frac{\sqrt{-1}}{2\pi}\partial\bar{\partial} \tilde{h}$ is $\mathcal{G}$-invariant, hence corresponds to an equivariant section.


We are now ready to give the definition of nefness.

\begin{definition}
Let $\mathcal{X}$ be a complex orbifold with underline space $X=|\mathcal{X}|$ compact and $\mathcal{L}\rightarrow \mathcal{X}$ a holomorphic line bundle on $\mathcal{X}$. We fix a Kähler form $\omega$ on $\mathcal{X}$. We say that $\mathcal{L}$ is \emph{nef}, if for any $\epsilon>0$, there exists a Hermitian metric $h_{\epsilon}$ on $\mathcal{L}$ such that its Chern curvature $\Theta_{h_{\epsilon}}$ satisfies
\[\Theta_{h_{\epsilon}}\geq -\epsilon \omega.\]
\end{definition}

\begin{rem}
Suppose that $\mathcal{L}\rightarrow \mathcal{X}$ is represented by $\mathcal{G}\ltimes L\rightarrow \mathcal{G}$. If $\mathcal{G}_0$ is compact, it is obvious that $L$ is then a nef line bundle on $\mathcal{G}_0$. However, in general $\mathcal{G}$ is not compact, and it makes no sense to say $L$ is nef or not.
\end{rem}

Let $\mathcal{X}$ be a complex orbifold, whose underlying space $X=|\mathcal{X}|$ is compact. We fix a Kähler metric $\omega$ on $\mathcal{X}$ and suppose $K_{\mathcal{X}}^{-1}$ is nef. We repeat the technique used by in \cite{DPS93} to construct a sequence of Kähler metrics $\{\omega_\epsilon\}\ $ in the same cohomology class of $\omega$ such that the Ricci form $\mathrm{Ricci}_{\omega_{\epsilon}}\geq -\epsilon \omega_{\epsilon}$.

For any $\epsilon>0$, since $K_{\mathcal{X}}^{-1}$ is nef, we have a Hermitian metric $h_{\epsilon}$ on $K_{\mathcal{X}^{-1}}$, such that $u_{\epsilon}=\Theta_{h_{\epsilon}}\geq -\epsilon \omega$. It is thus sufficient to search $\omega_{\epsilon}$ such that
\begin{equation}\label{f-m}
 \Ricci_{\omega_\epsilon}=-\epsilon \omega_\epsilon+\epsilon \omega +u_{\epsilon}.
\end{equation}

The $\partial \bar{\partial}$-lemma still holds in the orbifold setting.

\begin{lemma}[\emph{cf.} \protect{\cite[Theorem~H, Theorem~K]{Baily}}]
Let $(\mathcal{X},\omega)$ be a Kähler orbifold such that $X=|\mathcal{X}|$ is compact. If $\alpha$ is a $d$-exact $(p,q)$-form, then $\alpha$ is $\partial \bar{\partial}$-exact.
\end{lemma}

Hence we may write $u_{\epsilon}=\Ricci_{\omega}+\sqrt{-1}\partial \bar{\partial} f_{\epsilon}$.  And to search $\omega_{\epsilon}$ is the same as search a potential $\phi_\epsilon$ such that $\omega_\epsilon=\omega+ \sqrt{-1}\partial \bar{\partial}\phi_\epsilon$. \cref{f-m} on $\omega_\epsilon$ is thus equivalent to 
\begin{equation}\label{MA1}
\frac{(\omega+\sqrt{-1}\partial \bar{\partial}\phi_\epsilon)^n}{\omega^n}=\exp(\epsilon\phi_\epsilon-f_\epsilon)
\end{equation}

By the following theorem, \cref{MA1} has a unique solution.
\begin{thm}[Aubin-Yau Theorem]
\label{CY}
Let $(\mathcal{X},\omega)$ be a Kähler orbifold such that $X=|\mathcal{X}|$ is compact. For any smooth function $f$ on $\mathcal{X}$ and $\lambda>0$. The equation 
\begin{equation}
\tag{MA}
\label{MA}
\log \Monge(\phi)=\lambda \phi +f
\end{equation}
where $\Monge(\phi)=\frac{(\omega+\sqrt{-1}\partial\bar{\partial}\phi)^n}{\omega^n}$ is the Monge-Ampère operator, has a unique admissible solution.
\end{thm}
\begin{proof}
As it is well-known that \cref{MA} has a unique solution when $\lambda>0$ and a unique solution up to a constant when $\lambda=0$, we will just give a sketch and some references in the end for the interested readers.

The uniqueness follows from Hopf's maximum principle \cite[Chapter~3, Section~3, Theorem~8]{PW}, for the argument \emph{cf.} \cite[Proposition~7.13]{aubin82}. Note that \cite[Proposition 7.12]{aubin82} also points out that any $C^2$ solution to \cref{MA} will automatically be admissible. In fact, suppose $\phi$ is a $C^2$ solution and attains minimum at $x\in X=|\mathcal{X}|$, then in a local charts $\partial_i\partial_{\bar{j}}\tilde{\phi}$ is non-negative at $\tilde{x}$. Hence $\omega_{\phi}:=\omega+\sqrt{-1}\partial\bar{\partial}\phi$ is positive at $x$. Note $\frac{\omega_{\phi}^n}{\omega^n}=\exp(\lambda\phi+f)>0$. By continuity, $\omega_\phi$ can not have non-negative eigenvalue. Hence $\phi$ is admissible.

The existence of the solution is by the standard continuity method. Consider the following family of equations
\begin{equation}\tag{MAt}
\log \Monge(\phi)=\lambda \phi +tf
\label{MAt}
\end{equation}
where $t\in [0,1]$. Consider the set 
\begin{center}
$S:=\{t\in [0,1]:$ \cref{MAt} has a smooth solution $\phi_t\}$.
\end{center}
 We know $0\in S$, as $\phi_0=0$ is an obvious solution. \cref{MA} has a smooth solution is equivalent to $1\in S$. Our aim is to show $S$ is both open and closed. The openess follows from implicit function theorem and the closedness follows from an a priori $C^{3,\alpha}$-estimates. For the detail, we refer to \cite[Section~3 to 5]{Faulk19} and \cite[Chapter~3]{Sze}.
\end{proof}

Thus we have
\begin{lemma}
\label{sequence}
Let $\mathcal{X}$ be a complex orbifold with $X=|\mathcal{X}|$ compact and $-K_{\mathcal{X}}$ nef. Fix a Kähler metric $\omega$ on  $\mathcal{X}$. For $\epsilon>0$, there exists a Kähler metric $\omega_{\epsilon}$ cohomologous to $\omega$, and the Ricci form of $\omega_{\epsilon}$ satisfinging
\begin{equation*}
\mathrm{Ricci}_{\omega_{\epsilon}}\geq -\epsilon \omega_{\epsilon}
\end{equation*}
\end{lemma}

To prove our main results, we first note $\pi_1^{\mathrm{orb}}(\mathcal{X})$ is finitely generated. In fact we have

\begin{lemma}[\emph{cf.} \protect{\cite[Corollary ~1.2.5.] {Moe99}}]
\label{poincare}
Let $\mathcal{X}$ be an orbifold and $\mathcal{U}$ be an atlas of $\mathcal{X}$. There exists an atlas $\mathcal{V}$ for $\mathcal{X}$ such that
\begin{enumerate}
\item $\mathcal{V}$ refines $\mathcal{U}$;
\item For every chart $(\tilde{V},H,\psi)$ in $\mathcal{V}$, both $\tilde{V}$ and $V=\psi(\tilde{V})\subset |\mathcal{X}|$ are contractible;
\item The intersection of finitely many chart is empty or again a chart in $\mathcal{V}$.
\end{enumerate}
\end{lemma}

Let $\mathcal{X}$ be an orbifold with $X=|\mathcal{X}|$ compact. We may take a finite atlas $\mathcal{V}$ by \cref{poincare}. Note that each open sub-orbifold $[\tilde{V}/H]$ has fundamental group 
\begin{equation*}
\pi_1^{\mathrm{orb}}([\tilde{V}/H])\cong H.
\end{equation*}
By Van-Kampen theorem (\emph{cf.} \citep[Excercise~III.\protect{$\mathcal{G}$}.3.10]{Hae99}), we know that $\pi_1^{\mathrm{orb}}(\mathcal{X})$ is finitely generated.

We note the following lemma, which is proved in manifold case in \citep{DPS93}, holds true in the orbifold case with exactly the same proof.

\begin{lemma}[\emph{cf.} \protect{\cite[Lemma~1.3.]{DPS93}}]
\label{volume}
Let $U$ be a compact subset of $\tilde{X}=|\tilde{\mathcal{X}}|$. Then for any $\delta>0$, there exists a closed subset $U_{\epsilon,\delta}\subset U$ such that $\mathrm{vol}_{\omega}(U\setminus U_{\epsilon,\delta})<\delta$ and $\mathrm{diam}_{\omega_\epsilon}(U_{\epsilon,\delta})<C_1 \delta^{\frac{1}{2}}$, where $C_1$ is a constant independent of $\epsilon$ and $\delta$.
\end{lemma}

\begin{thm}[=\cref{mt}]
\label{Mt} 
 Let $(\mathcal{X},\omega)$ be a Kähler orbifold such that $X=|\mathcal{X}|$ is compact. If the anti-canonical bundle $K_{\mathcal{X}}^{-1}$ is nef, then $\pi_1^{\mathrm{orb}}(\mathcal{X})$ has polynomial growth.
\end{thm}
\begin{proof}
We could reproduce the argument by \cite{DPS93} and \cite{Paun97} in manifold case.
Let $\omega_\epsilon$ be the sequence of Kähler metrics we obtained in \cref{sequence}, and $\tilde{\mathcal{X}}\rightarrow \mathcal{X}$ be the universal covering of $\mathcal{X}$. We fix a finite system of generators $\{\gamma_i\}$ of $\pi_1^{\mathrm{orb}}(\mathcal{X})$. Hence to prove the main theoreit suffices to show there exists $\omega'$ such that for all $i$, $\gamma_i\in \Gamma_{\alpha}$ in \cref{mag}.

Take a compact subset $U\subset \tilde{X}$ which contains the fundamental domain $F$ of $\pi_1^{\mathrm{orb}}(\mathcal{X})$. As $\{\gamma_i\}$ is finite, we may take $U$ large enough, such taht $U\cap \gamma_j U\neq \emptyset$ for all $j$. 
We choose $\delta >0$ sufficient small such that $\delta <\frac{1}{4} \vol_{\omega}(U\cap\gamma_j U)$ and $\delta < \frac{1}{2} \vol_\omega F$. By \cref{volume}, there exists a subset $U_{\epsilon,\delta}\subset U$, such that $\diam_{\omega_{\epsilon}}(U_{\epsilon,\delta})<C_1\delta^{\frac{1}{2}}:=C$. 
By the choice of $\delta$, we know that $U_{\epsilon,\delta}\cap \gamma_j U_{\epsilon,\delta}\neq \emptyset$. Take $\tilde{x}\in U_{\epsilon,\delta}$, we know that $d_{\omega_\epsilon}(\tilde{x},\gamma_j\tilde{x})<C$.
We set $\tilde{\omega_\epsilon}:=\frac{\epsilon}{n-1}\omega_\epsilon$. Then $\Ricci_{\tilde{\omega_\epsilon}}\geq -(n-1)\tilde{\omega_\epsilon}$ and $\mathrm{d}_{\tilde{\omega_\epsilon}}(\tilde{x},\gamma_j\tilde{x})<\frac{\epsilon}{n-1}C$.
For $\epsilon$ sufficient small, we see that $\frac{\epsilon}{n-1}C<\alpha$. Hence by \cref{mag}, we know $\pi_1^{\mathrm{orb}}(\tilde{X})$ is virtually nilpotent. By Gromov's theorem on polynomial growth (\emph{cf.} \citep[Main Theorem]{Gro}), we know that for a finitely generated group $G$, $G$ is of polynomial growth is equivalent to $G$ is virtually nilpotent. Hence $\pi_1^{\mathrm{orb}}(\tilde{X})$ has polynomial growth
\end{proof}

\section{a further result}
\label{sec6}
For a complex orbifold $\mathcal{X}$ of dimension $n$, its orbifold structure gives rises to a klt pair $(X,\Delta_X)$. It turns out, the pair complete determines $\mathcal{X}$. We first recall the definition.
\begin{definition}[\emph{cf.}\protect{\citep[Definition~3.1]{Campana_Claudon14}}]
\label{pair}
A klt pair $(X,\Delta)$ is an orbifold pair if $\Delta$ is a $\mathbb{Q}$-Weil divisor of the form 
\begin{equation*}
\Delta=\sum (1-\frac{1}{m_i})D_i,
\end{equation*}
where $m_i\geq 2$ are integers and $(X,\Delta)$ satisfies the locally uniformizable condition:\newline
there exists finite morphisms $\phi_j:U_j\rightarrow X$ such that
\begin{enumerate}
\item $\phi_j(U_j)\subset X$ is open and $\bigcup\phi_j(U_j)=X$;
\item $\phi_j:U_j\rightarrow \phi_j(U_j)$ is a Galois analytical cover and it's branching divisor $B(\phi_j)=\Delta|_{\phi_j(U_j)}$
\end{enumerate}
\end{definition}

For the basics of orbifold paris we refer to \citep{G-K07} section 3 and section 4, \citep{Campana_Claudon14} section 3. We emphasise that an orbifold pair $(X,\Delta)$ induces a unique complex orbifold structure on the underlying topological space $X_{\mathrm{top}}$ and vice versa.  We note that the complex space structure on $|\mathcal{X}|$ is induced via the holomorphic quotient theorem of Cartan (\emph{cf.} \citep[Theorem~3.8.13]{Bar}) and that for a Galois analytical covering $\pi:Y\rightarrow X$, the branching divisor $B(\pi)$ is defined so that we have the equation on $\mathbb{Q}$-Weil divisor classes
\begin{equation}
\label{divisor}
K_Y=\pi^{\ast}(K_X+B).
\end{equation}

Suppose $\mathcal{X}$ is a complex orbifold of dimension $n$ which is represented by an orbifold groupoid $\mathcal{G}=\{G_1\rightrightarrows G_0\}$ and $(X=|\mathcal{X}|,\Delta)$ its associated orbifold pair. For any $x\in G_0$ by \cref{local}, there exists open neighborhood $x\ni U_x\subset G_x$, such that $\pi_x:U_x\rightarrow U_x/G_x\subset |\mathcal{X}|$ gives an orbifold chart. By holomorphic quotient theorem, $U_x/G_x$ has a unique   normal complex space structure such that $\pi_x:U_x\rightarrow U_x/G_x$ is a Galois analycial covering. As we mentioned above, the complex space structure of $|X|$ comes from the holomorphic quotient theorem. Hence $U_x/G_x\subset X$ is an open sub-variety and $\pi_x:U_x\rightarrow U_x/G_x$ is a local uniformization as in \cref{pair}. We also know $\pi:G_0\rightarrow X$ is holomorphic. As $G_0$ is Kähler, by \citep[Proposition~3.3.1]{Var89}, $X$ is a complex Kähler space. We consider the canonical bundle $K_{\mathcal{X}}$. From \cref{hol}, we know it is represented by $\mathrm{det}(\Omega_{G_0})$. We denote by $K_{G_0}$ its canonical class. Then by \cref{divisor}, we have 
\begin{equation}
\label{divisor2}
K_{G_0}|_{U_x}=\pi_x^{\ast}(K_X+\Delta).
\end{equation}
If $U_x/G_x\subset U_y/G_y$, we have an element $g\in G_1$ which induces an embedding $\rho_g:U_x\rightarrow U_y$ and we have the following commutative diagram:
\begin{center}
$\xymatrix{U_x \ar[r]^{\rho_g}\ar[d]^{\pi_x} & U_y \ar[d]^{\pi_y}\\
U_x/G_x\ar[r] & U_y/G_y}$
\end{center}
Thus the local equations \cref{divisor2} glue together to 
\begin{equation*}
K_{G_0}=\pi^{\ast}(K_X+\Delta)
\end{equation*}
We may thus regard $K_X+\Delta$ as the canonical class of $(X,\Delta)$. As $(X,\Delta)$ has klt singularities, we know there exists a positive integer $a(X)$ such that $a(X)(K_X+\Delta)$ is a Cartier divisor, hence we have the line bundle $\mathcal{O}_X(a(X)(K_X+\Delta))$ on $X$. We thus have $\mathcal{O}_{G_0}(K_{G_0})^{\otimes a(X)}=\pi^{\ast}\mathcal{O}_X(a(X)(K_X+\Delta))$. It $h$ is an Hermitian metric on $K_{\mathcal{X}}$, then the Hermitian metric $h^{\otimes a(X)}$ on $K_{\mathcal{X}}^{\otimes a(X)}$ induces an Hermitian metric on $\mathcal{O}_X(a(X)(K_X+\Delta))$, since $h$ is $G_1$-invariant. On the other hand, the pullback of an Hermitian metric on $\mathcal{O}_X(a(X)(K_X+\Delta))$ by $\pi$ will induce an Hermitian metric on $K_{\mathcal{X}}^{\otimes a(X)}$.
In particular, if $X$ is compact, $K_{\mathcal{X}}^{-1}$ being nef is equivalent to $-(K_X+\Delta)$ being nef.

\begin{definition}
\label{cpi1}
Let $(X,\Delta)$ be a klt pair with $\Delta=\sum (1-\frac{1}{m_i})D_i$ where $m_i\geq 2$ are integers. We define its fundamental group $\pi_1(X,\Delta)$ by 
\begin{equation*}
\pi_1(X_{\mathrm{reg}}\setminus|\Delta|)/N,
\end{equation*}
where $N$ is the normal group generated by $\gamma_j^{m_j}$, with $\gamma_j$ a small loop around the component $D_j$ with multiplicity $m_j$ of $\Delta$.
\end{definition}
By \citep[lemme~1.9.9]{cam01}, we know that when $(X,\Delta)$ is an orbifold pair, we have a short exact sequence
\begin{center}
$1\rightarrow K\rightarrow \pi_1(X,\Delta)\rightarrow \pi_1(X)\rightarrow 1$
\end{center}
with  $K$ generated by torsion elements.

It is a well-known fact that if $\mathcal{X}$ is a complex orbifold with $(X,\Delta)$ its associated pair, then there is a canonical isomorphism $\pi_1(X,\Delta)\rightarrow \pi_1^{\mathrm{orb}}(\mathcal{X})$.

We have the following result.
\begin{thm}[=\cref{mt2}]
\label{abe}
Let $\mathcal{X}$ be a compact Kähler orbifold with $-K_{\mathcal{X}}$ nef. Let $(X,\Delta)$ be the associated orbifold pair of $\mathcal{X}$. If $X$ is projective, then $\pi_1^{\mathrm{orb}}(\mathcal{X})$ is virtually Abelian.
\end{thm}
\begin{proof}
By \cref{mt}, we know there exists a nilpotent subgroup $\Gamma < \pi^{\mathrm{orb}}(\mathcal{X})$ such that its index $[\pi^{\mathrm{orb}}(\mathcal{X}):\Gamma]$ is finite. By \cref{Galois}, we may find a finite cover $\tilde{\mathcal{X}}$ of $\mathcal{X}$ such that $\pi_1^{\mathrm{orb}}(\tilde{\mathcal{X}})=\Gamma$. For any $\epsilon>0$ and the Hermitian metric $h_{\epsilon}$ such that $\Theta_{h_{\epsilon}}\geq \epsilon \omega$, we may take induced metric on $-K_{\tilde{\mathcal{X}}}$. In particular, we have $-K_{\tilde{\mathcal{X}}}$ nef.

To simplify the notation, we may assume that $\pi_1^{\mathrm{orb}}(\mathcal{X})$ is nilpotent. Consider $(X,\Delta)$, we know that $-(K_X+\Delta)$ is nef. By \citep[Corollary~2.]{Zhang05}, we know the Albanese map $\alpha_X:X\rightarrow \mathrm{Alb}(X)$ is dorminant. Recall that  the Albanese map of $X$ is defined by the Albanese morphism of its smooth model. Let $r:Y\rightarrow (X,\Delta)$ be a log resolution, we have
\begin{center}
$\xymatrix{Y \ar[r]^r\ar[d]_{\alpha_{Y}} & X \ar@{.>}[dl]^{\alpha_X}\\
\mathrm{Alb}(Y)=:\mathrm{Alb}(X)}$
\end{center}
As $X$ has rational singularities, by \citep[Propsition~2.3.]{Reid83} $\alpha_X$ is defined on all $X$ and hence surjective. We have
\begin{equation*}
K_Y+r_{\ast}^{-1}(\Delta)=r^{\ast}(K_X+\Delta)+\sum E_j,
\end{equation*}
where $E_j\subset \mathrm{exc}(r)$ are the prime exceptional divisors. We have 
\begin{equation}
\label{space}
Y\setminus(r_{\ast}^{-1}(\Delta)\cup \mathrm{exc}(r))=X_{\mathrm{reg}}\setminus |\Delta|.
\end{equation}
If $l_j$ is a loop in $Y$ around $E_j$, we have $r\circ l_j$ will be a loop around some $D_i$. By the definition of $\pi_1(X,\Delta)$, we could find a large enough $n_j$ such that $(r\cdot l_j)^{n_j}$ is the unit in $\pi_1(X,\Delta)$. We define 
\begin{equation*}
\Delta_Y=r_{\ast}^{-1}(\Delta)+\sum (1-\frac{1}{n_j})E_j.
\end{equation*}
By \cref{space}, the choice of $n_j$ and \cref{cpi1}, we have $\pi_1(Y,\Delta_Y)=\pi_1(X,\Delta)$. As $r:Y\rightarrow X$ is a log resolution, $|\Delta_Y|$ is snc. If $r_{\ast}^{-1}(D_i)$, $1\leq i\leq i_0$ and $E_j$, $1\leq j\leq j_0$ are all the components of $\Delta_Y$ passing through a point $y\in Y$. Then there is holomorphic chart $(U,\phi)$ centered at $y$, such that $\phi_{\ast}(\Delta_Y|_U)$ is the branching divisor of the following map
\begin{center}
$(z_1,...,z_{i_0},z_{i_0+1},...,z_{i_0+j_0},...,z_n)\mapsto (z_1^{m_1},...,z_{i_0}^{m_{i_0}},z_{i_0+1}^{n_1},...,z_{i_0+j_0}^{n_j},z_{i_0+j_0+1},...,z_n)$
\end{center}
Hence $(Y,\Delta_Y)$ is an orbifold pair. We write the short exact sequence 
\begin{equation*}
1\rightarrow K\rightarrow \pi_1(Y,\Delta_Y)\rightarrow \pi_1(Y)\rightarrow 1
\end{equation*}
Recall that (by \citep[lemme~1.9.9]{cam01}) $K$ is generated by torsion elements. As $\pi_1(Y,\Delta_Y)=\pi_1(X,\Delta)$ is nilpotent, we have $\pi_1(Y)$ is also nilpotent. By \citep{hir38}, for any nilpotent group $N$ of finite type, the torsion element of $N$ forms a finite normal subgroup $N_{\mathrm{tor}}\trianglelefteq N$ and the nilpotent limit of $N$ is $N/N_{\mathrm{tor}}$. By the above exact sequence, we have \begin{equation*}
\pi_1(Y,\Delta_Y)/\pi_1(Y,\Delta_Y)_{\mathrm{tor}}=\pi_1(Y)/\pi_1(Y)_{\mathrm{tor}}.
\end{equation*}
As $\alpha_Y$ is surjective, by \citep[Théorèm~2.2]{cam98}, we have $\pi_1(Y)/\pi_1(Y)_{\mathrm{tor}}=\pi_1(\mathrm{Alb}(X))$. By \citep[Lemme~A.0.1]{claudon07}, we know this means $\pi_1(Y,\Delta_Y)$ is virtually Abelian. As $\pi_1(Y,\Delta_Y)$ has finite index in $\pi_1^{\mathrm{orb}}(\mathcal{X})$, we know that $\pi_1^{\mathrm{orb}}(\mathcal{X})$ is also virtually Abelian.
\end{proof}

\begin{rem}
As one can see from the proof of \cref{abe}, the hypothesis $X$ is projective is used only to show that the Albanese morphism is surjective. One may reformulate \cref{abe} as following:

If the covering $(X',\Delta')$ corresponds to the nilpotent subgroup of $\pi_1(X,\Delta)$ has surjective Albanese morphism, then $\pi_1(X,\Delta)$ is virtually Abelian.
\end{rem}

\bibliographystyle{alpha}

\bibliography{orbi-fundamental}

\end{document}